\documentclass{article}
\usepackage[utf8]{inputenc}
\usepackage[english]{babel}
\usepackage{graphicx}
\usepackage[svgnames]{xcolor}
\usepackage{indentfirst}
\usepackage[colorlinks]{hyperref}
\usepackage{listings} 
\usepackage{amsmath}
\usepackage{amssymb}
\usepackage{amsthm} 
\usepackage{graphicx,geometry}
\usepackage{enumitem}
\usepackage{subfig}
\usepackage{mathtools}
\usepackage{float}
\usepackage{wrapfig}
\usepackage{lpic}
\usepackage{tikz}
\usetikzlibrary{math,patterns}

\newtheorem{theorem}{Theorem}
\newtheorem{lemma}{Lemma}
\newtheorem{proposition}{Proposition}
\newtheorem{corollary}{Corollary}

\theoremstyle{definition}

\newtheorem{definition}{Definition}
\newtheorem{remark}{Remark}

\usepackage[title]{appendix}

\begin{document}

\title{Impact of dissipation ratio on vanishing viscosity solutions of the Riemann problem for chemical flooding model}

\author{F. Bakharev\footnote{Chebyshev Laboratory, St. Petersburg State University, 14th Line 29b, 199178,  Saint-Petersburg, Russia. E-mail: f.bakharev@spbu.ru.},
A. Enin\footnote{Chebyshev Laboratory, St. Petersburg State University, 14th Line 29b, 199178 Saint-Petersburg, Russia. E-mail: a.enin@spbu.ru.},
Yu. Petrova\footnote{Chebyshev Laboratory, St. Petersburg State University, 14th Line 29b, 199178,  Saint-Petersburg, Russia. 
Instituto de Matemática Pura e Aplicada, Estrada Dona Castorina, 110, 
Jardim Botanico | CEP 22460-320, 
Rio de Janeiro, RJ - Brasil. 
E-mail: j.p.petrova@spbu.ru, yulia.petrova@impa.br },
N. Rastegaev\footnote{Leonhard Euler International Mathematical Institute (SPbU Department), 14th Line 29b, 199178, Saint-Petersburg, Russia. E-mail:  n.v.rastegaev@spbu.ru.}}
\renewcommand{\today}{}
\maketitle
\abstract{
The solutions for a Riemann problem arising in chemical flooding models are studied using vanishing viscosity as an admissibility criterion. We show that when the flow function depends non-monotonically on the concentration of chemicals, non-classical undercompressive shocks appear. The monotonic dependence of the shock velocity on the ratio of dissipative coefficients is proven. For that purpose we provide the classification of the phase portraits for the travelling wave dynamical systems and analyze the saddle-saddle connections.

\medskip


}

\maketitle

\section{Introduction}
\label{sec:intro}
In this paper we study the non-uniqueness of vanishing viscosity admissible solutions for a Riemann problem for the system of conservation laws ($x\in\mathbb{R}$, $t>0$):
\begin{equation}\label{eq:main_system}
\begin{cases} 
s_t + f(s, c)_x = 0, \\
(cs + a(c))_t + (cf(s,c))_x  = 0,
\end{cases}
\end{equation}
with initial data 
\begin{equation}
\label{eq:Riemann-problem}
    (s,c)(x,0)=
    \begin{cases}
        (s^L,c^L),& \text{if } x\leq 0,\\
        (s^R,c^R),& \text{if } x>0.
    \end{cases}
\end{equation}
This system is often used to describe the displacement of oil in porous media by a hydrodynamically active chemical agent (polymer, surfactant, etc). Thus, the water phase saturation $s=s(x,t)$ and the chemical agent concentration $c=c(x,t)$ are assumed to take values in the interval $[0,1]$. These upper and lower bounds for $s$ and $c$ in fact follow from the structure of the water fractional flow function $f$ (the Buckley–Leverett function~\cite{BL1942}), which is commonly assumed to be $S$-shaped of $s$ for every $c$ (see full list of assumptions (F1)-(F4) in Section \ref{subsec:restrictions}). In this paper for simplicity we study a specific Riemann problem $(s^L, c^L)=(1, 1)$ and $(s^R, c^R)=(0, 0)$ and discuss some possible generalizations in Section \ref{sec:generalisations}. Function $a$ describes the adsorption of the chemical agent and is assumed to be increasing and concave.

Consider the dissipative analogue of system \eqref{eq:main_system}, see e.g.~\cite[Chapter 5]{TheBook}:
\begin{equation}\label{eq:main_system_dissipation}
\begin{cases} 
s_t + f(s, c)_x = \varepsilon_c (A(s,c)s_x)_x, \\
(cs + \alpha)_t + (cf(s,c))_x  = \varepsilon_c (c A(s,c)s_x)_x+\varepsilon_d c_{xx},\\
\varepsilon_r \alpha_t = a(c)-\alpha,
\end{cases}
\end{equation}
where $\varepsilon_c$, $\varepsilon_d$ and $\varepsilon_r$ are the dimensionless capillary pressure, diffusion, and relaxation time, respectively; $A(s,c)$ is the capillary pressure function (bounded and separated from zero) and $\alpha=\alpha(x,t)$ is the dynamic adsorption term. In current paper for simplicity we  will consider the situation with $\varepsilon_d$ or $\varepsilon_r$ equal to zero. Some discussion concerning the case with three non-zero dissipative parameters can be found in Section \ref{subsec_6_2}. Let us note that parameters $\varepsilon_c$, $\varepsilon_d$ and $\varepsilon_r$ are usually assumed to be small because of their nature. These coefficients appear after rescaling the corresponding system in domain of finite size and contain the distance between inlet and outlet in the denominator. Tending this distance to infinity leads us to the Riemann problem and yields simultaneous tending of three coefficients to zero while their ratios remain constant. We say that discontinuities (shocks) in solutions of \eqref{eq:main_system} are admissible if they are the limits of solutions of \eqref{eq:main_system_dissipation} when dimensionless groups $\varepsilon_c$, $\varepsilon_d$, and $\varepsilon_r$ tend to zero. 

The case when fractional flow function $f$ depends monotonically on $c$ was studied in detail by Johansen \& Winther in~\cite{JohansenWinther} using vanishing viscosity criterion for shocks. Note that in the monotone case this admissibility criterion coincides with the other well-known criterion:  the shock is admissible if and only if  the Rankine-Hugoniot, Lax and Oleinik conditions are fullfilled (see \cite[Appendix A]{TheBook}). For non-monotonic dependence of $f$ on chemical agent $c$ these admissibility criteria are not equivalent anymore. Moreover, as it was noticed previously in paper \cite{EntovKerimov} by Entov \& Kerimov, the Lax admissibility criterion gives non-physical solutions, i.e. the agent may have no effect on the flow, which contradicts fundamental expectations (for details see example in Section~\ref{subsec:example}).

The case when flow function $f$ depends non-monotonically on $c$ was the subject of study by Shen~\cite{Shen}. In that work the system \eqref{eq:main_system} was first decoupled in Lagrangian coordinates and then regularized by adding small diffusion terms, so it gives no sufficient answer about physically meaningful viscosity solutions to the initial system. We also wish to attract more attention to the paper of Entov~\&~Kerimov~\cite{EntovKerimov}. 
Both~\cite{EntovKerimov} and~\cite{Shen} notice, that vanishing viscosity admissibility criterion 
gives solutions containing non-classical undercompressive shocks (known also as transitional waves) that depend on the ratio of small coefficients $\varepsilon_c/\varepsilon_d$ in smoothing terms.  The main aim of current paper is to generalize the ideas introduced in~\cite{EntovKerimov} to consider non-equilibrium adsorption and a wider class of non-monotonous functions $f$ and to give mathematically rigorous formulations and proofs to the assertions, which~\cite{EntovKerimov} lacks.

In the last decades there is an extensive research on undercompressive shocks for hyperbolic systems of conservation laws. They violate the standard Lax admissibility criteria and correspond to saddle-to-saddle connection for travelling wave solutions. An important feature of such waves is the sensitivity of the solution to the diffusion/dispersion terms. They appear in various contexts: balance of diffusion and dispersion in elasticity of nonlinear materials and phase transition dynamics (see survey~\cite{LeFloch1999} and book~\cite{LeFloch2002} and references therein, especially bibliographical notes), thin films theory (see e.g.~\cite{Bertozzi1999}, also a recent work on tears of wine~\cite{Bertozzi2020}), oil recovery (e.g. three-phase flow, oil-water-gas~\cite{WAG2001}), flow of pedestrians~\cite{pedestrian2005}. Transitional waves make the structure of the solution to a Riemann problem richer, serving as a bridge which joins the characteristic families (see e.g.~\cite{ScMP1996},~\cite{ELI1990}).

The paper has the following structure. Section~\ref{sec:problem-statement} contains precise assumptions on the system~\eqref{eq:main_system} and review of basic notions needed for construction of the solution to a Riemann problem. Also in~Section~\ref{subsec:example} we consider a simple example when Lax admissibilty criterion is inappropriate.  In Section~\ref{sec:Riemann-solution} we formulate the main result of the paper  (Theorem~\ref{Theorem1}). In Section~\ref{sec:shock-waves} we give detailed description of the dynamical system for travelling wave solutions of~\eqref{eq:main_system_dissipation} and give its detailed phase portrait description. In Section~\ref{sec:shock-admissibility} we give the proof  of the main result. Section~\ref{sec:generalisations} is devoted to discussions and possible generalizations of Theorem~\ref{Theorem1}. In Appendix~\ref{ap:A} we give a proof to some basic properties of trajectories for a travelling wave dynamical system.
\section{Problem statement}
\label{sec:problem-statement}

\subsection{Properties of the functions}
\label{subsec:restrictions} The following assumptions (F1)--(F4) for the fractional flow function $f$ are formulated in a strong form to avoid more technical details in proof (see Fig.\ref{fig:BL_ads}a for an example of function~$f$). Nevertheless, they can be weakened (see Section~\ref{sec:generalisations} for more detailes).
\begin{enumerate}
    \item[(F1)] $f\in C^2([0,1]^2)$; $f(0, c)=0$; $f(1, c)= 1$;
    \item[(F2)] $f_s(s, c)>0$ for $0<s<1$, $0 \leq c \leq 1$;  $f_s(0,c)=f_s(1,c)=0$;
    \item[(F3)] $f$ is $S$-shaped in $s$: for each $c \in [0,1]$ function $f(\cdot,c)$ has a unique point of inflection $s^I =  s^I(c) \in (0, 1)$, such that $f_{ss}(s, c)>0$ for $0<s<s^I$ and $f_{ss}(s, c)<0$ for $s^I<s<1$. 
    \item[(F4)] $f$ is non-monotone in $c$: $\forall s \in (0,1) \, \exists c^*(s) \in (0,1)$:
    \begin{itemize}
        \item $f_c(s, c)<0$ for $0<s<1$, $0 < c < c^*(s)$;
        \item $f_c(s, c)>0$ for $0<s<1$, $c^*(s) < c < 1$.
    \end{itemize}
    Function $c^*(s)$ clearly solves $f_c(s, c^*(s)) = 0$, thus $c^*(s)$ is continuous by the variation of the implicit function theorem.
\end{enumerate}
\begin{figure}[htbp]
    \centering
    \includegraphics[width=0.42\textwidth]{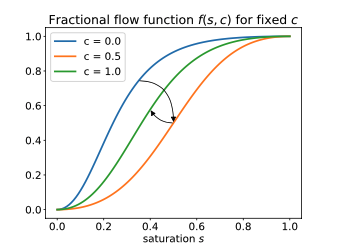}
    \includegraphics[width=0.42\textwidth]{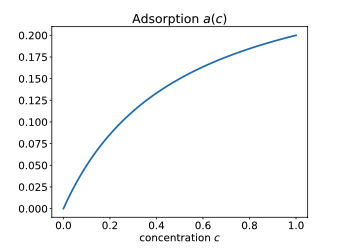}\\
    (a)\qquad\hfil\qquad (b)\hfil
    \caption{(a) example of non-monotonic dependence of $f(s,c)$ on $c$; (b) adsorption function}
    \label{fig:BL_ads}
\end{figure}

The adsorption function $a=a(c)$ is such that (see Fig.\ref{fig:BL_ads}b):
\begin{itemize}
    \item[(A1)] $a \in C^2[0,1]$, $a(0) = 0$;
    \item[(A2)] $a'(c)>0$ for $0<c<1$;
    \item[(A3)] $a''(c)<0$ for $0<c<1$.
\end{itemize}

\subsection{Review of the main notions for the Riemann problem  solution}
\label{subsec:basic-notions}

In this section we will briefly recall the main notions from~\cite{JohansenWinther} required for the construction of the solution to the Riemann problem. The general information can be found in \cite[Chapter~9]{Dafermos}; for more detailed analysis of the system~\eqref{eq:main_system} see~\cite{JohansenWinther}.

\begin{definition}
\textit{A shock wave} between points $u^-=(s^-,c^-)$ and $u^+=(s^+,c^+)$ with velocity $v$ is a discontinuous solution $u=(s,c)$ of~\eqref{eq:main_system}, defined in weak form:
\begin{align*}
    u(x,t)=
    \begin{cases}
        u^-,& x<v t,
        \\
        u^+,& x>v t,
    \end{cases}
\end{align*}
which satisfies the Rankine-Hugoniot (RH) conditions:
\begin{equation}
\label{eq:RH-1}
\begin{split}
    v[s]&=[f(s,c)],
    \\
    v[cs+a(c)]&=[cf(s,c)],
\end{split}
\end{equation}
and some admissibility criteria. Here $[q(s,c)]=q(s^-,c^-)-q(s^+,c^+)$. In this paper we assume vanishing viscosity admissibility criterion, which is explained in detail in Section \ref{sec:Riemann-solution}. \end{definition}
\begin{definition}
 A shock is called a \textit{$c$-shock} if $c^-\not=c^+$ and denoted by $u^- \xrightarrow{c\text{-shock}}u^+$.
\end{definition}

\begin{definition}
 A solution to a Riemann problem~\eqref{eq:main_system}--\eqref{eq:Riemann-problem} is called an \textit{$s$-wave} if $c^L=c^R$ and denoted by $u^L \xrightarrow{s} u^R$.
\end{definition}
For an $s$-wave the concentration $c$ stays constant for all times and the system \eqref{eq:main_system} reduces to a scalar conservation law with a well-known solution~\cite{Lax1957,Gelfand1959}. 
For $s^L>s^R$ the unique solution of the Riemann problem is given by
\begin{align*}
    s(x,t)=
    \begin{cases}
        s^L,& x/t<g(s^L),
        \\
        s, & x/t=g(s),
        \\
        s^R, & x/t>g(s^R),
    \end{cases}
\end{align*}
where $g(s)=\frac{\partial}{\partial s} f_U(s,c)$ and $f_U$ is the upper convex envelope of $f$ with respect to the interval $[s^L,s^R]$. 
The velocity of $s$-wave $u^L \xrightarrow{s} u^R$ at points $u^L$ and $u^R$ is equal to $g(s^L)$ and $g(s^R)$, respectively. 

Consider two waves $u_1\xrightarrow{a} u_2$ and $u_2\xrightarrow{b} u_3$. Let $v_f^a$ denote the final wave speed of the $a$-wave (wave speed at point $u_2$) and $v_i^b$ the initial wave speed of the $b$-wave (wave speed at point $u_2$). If $a$ is a shock wave with velocity $\sigma$, then $v_i^a=v_f^a=\sigma$. 
\begin{definition}
We say that a pair of waves $u_1\xrightarrow{a} u_2$ and $u_2\xrightarrow{b} u_3$ is compatible by speed if their combination $u_1\xrightarrow{a} u_2 \xrightarrow{b} u_3$ solves the Riemann problem with left state $u_1$ and right state $u_3$. Thus the two waves are compatible if and only if $v_f^a\leqslant v_i^b$.
\end{definition}


The solution of any Riemann problem~\eqref{eq:Riemann-problem} consists of a sequence of  waves that are compatible by speed and connect the given left state $u^L$ with the given right state $u^R$ (see \cite[Chapter~9]{Dafermos}).
The following proposition is a straightforward consequence of \cite[Lemma 5.1]{JohansenWinther} for the case $c^L> c^R$ (note that Lemma~5.1 does not assume the monotonicity of $f$  and uses only the concavity of $a$):

\begin{proposition}
\label{prop:solution-RP}
There exists $u^-=(s^-,c^L)$ and $u^{+}=(s^+,c^R)$ such that the solution to the Riemann problem~\eqref{eq:main_system}--\eqref{eq:Riemann-problem} has the following structure:
\begin{align}
\label{eq:solution-RP-1}
    u^L\xrightarrow{s} u^{-} \xrightarrow{c\text{-shock}} u^{+} \xrightarrow{s} u^R.
\end{align}
\end{proposition}

The main consequence of the  Proposition~\ref{prop:solution-RP} in our case is that the construction of the solution to the Riemann problem~\eqref{eq:main_system}--\eqref{eq:Riemann-problem} is reduced to finding an ``appropriate'' $c$-shock wave, i.e. the states $u^-$, $u^+$ and the velocity of the shock $v$. There exist many admissibility criteria for shocks (see e.g. \cite[Chapter 8]{Dafermos}). As we demonstrate in Section~\ref{subsec:example} below, not all of them give physically meaningful solutions.

\subsection{Motivating example}
\label{subsec:example}

Consider the simplest model with non-monotone flow function (we call it the ``boomerang'' model): the fractional flow function $f$ decreases in $c$ from $c=0$ up to some value $c^M\in(0,1)$, and then increases from $c^M$ to $c=1$ back to the same function, i.e. $f(s,1)=f(s,0)$. For example (Fig.~\ref{fig:BL_RS_Lax}a),
\begin{align*}
    f(s,c) = \frac{s^2}{s^2+\mu(c)(1-s)^2}, \qquad \mu(c)=1+4c(1-c), \qquad c^M=0.5.
\end{align*}
In the monotone case the celebrated Lax admissibility condition~\cite[\S8.3]{Dafermos}, alongside with conditions (F1)--(F3) and (A1)--(A3), is equivalent to the vanishing viscosity criterion for the dissipative system \eqref{eq:main_system_dissipation} (with $\varepsilon_d=0$), for details see~\cite[Appendix A]{TheBook}.


The straightforward application of Lax criterion for the ``boomerang'' model gives the solution~$s$ that doesn't reflect any change of flow function~$f$ in the interval $c\in(0,1)$. As we can see from Fig.~\ref{fig:BL_RS_Lax}b the solution $s=s(x,t)$ coincides with the solution of scalar conservation law as if polymer concentration doesn't change, which contradicts the physical intuition. This observation motivates us to step back from the Lax admissibility condition and turn to a more physically appropriate admissibility criterion, i.e. vanishing viscosity criterion. 

\begin{figure}[H]
    \centering
    \includegraphics[width=0.42\textwidth]{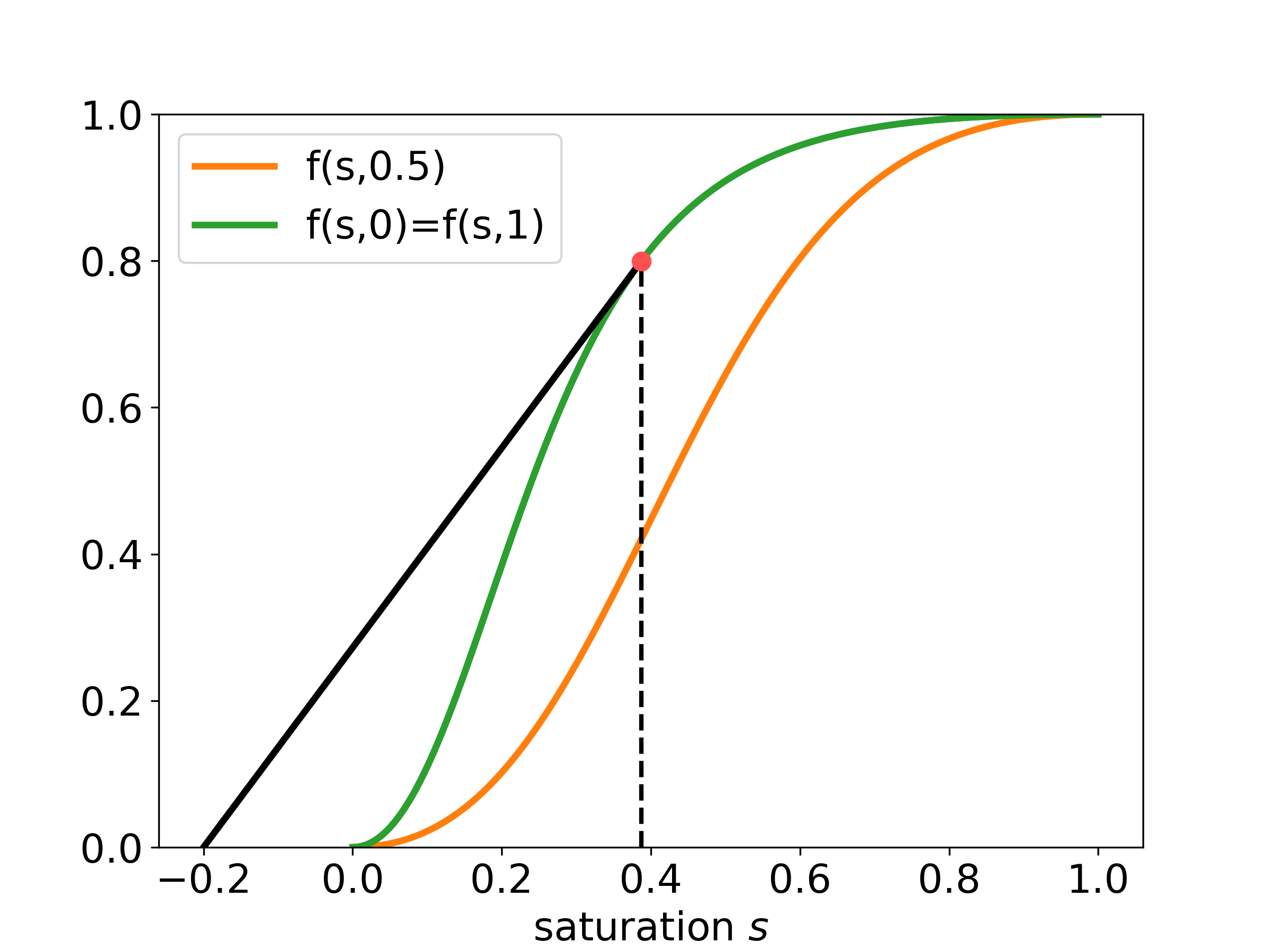}
    \includegraphics[width=0.42\textwidth]{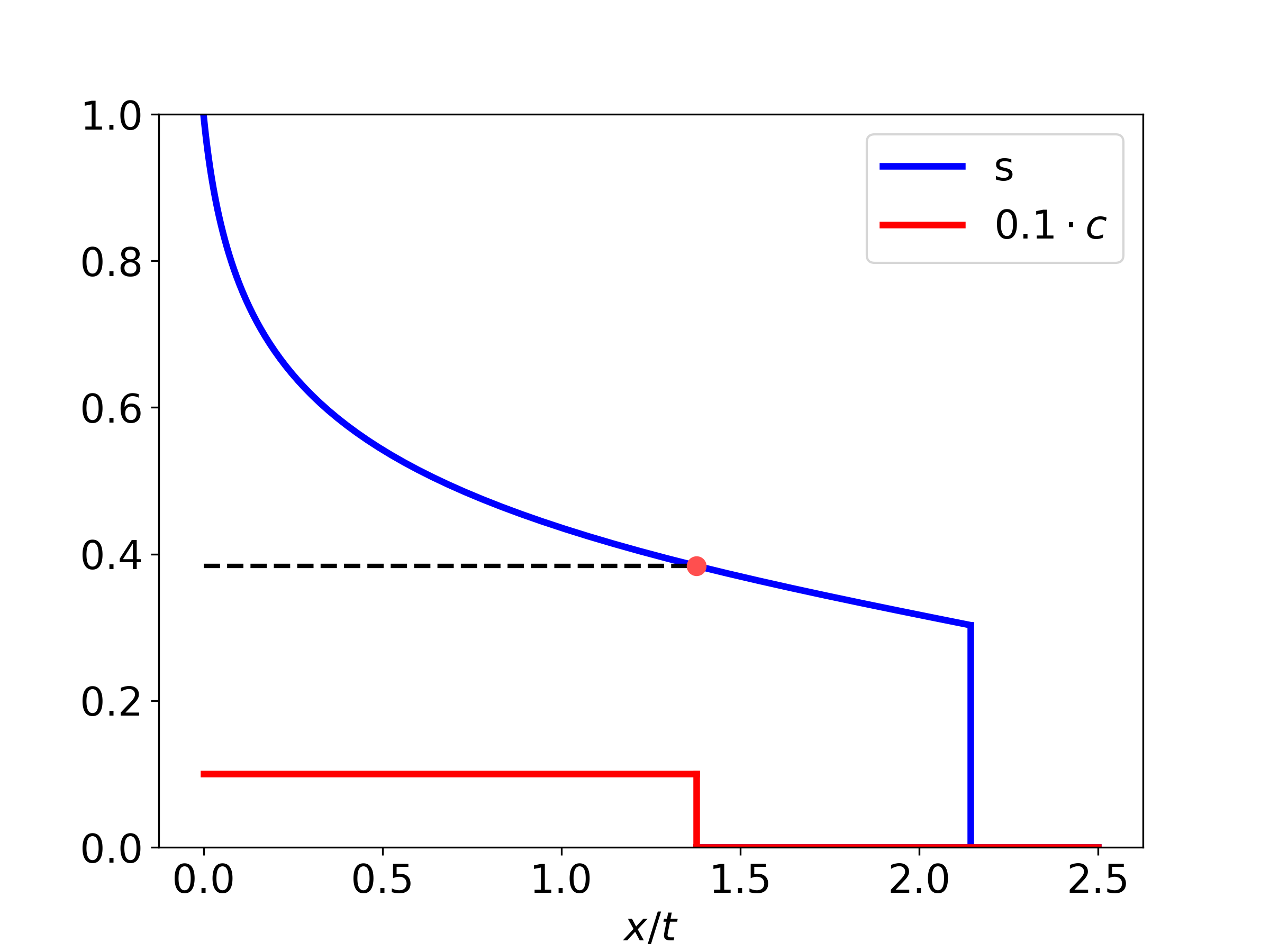}\\
    (a)\qquad\hfil \qquad(b)\hfil
    \caption{``boomerang'' model: (a) $(s,f)$ diagram; 
    (b) solution to a Riemann problem for Lax admissibility criterion}
    \label{fig:BL_RS_Lax}
\end{figure}

\section{Formulation of the main result}
\label{sec:Riemann-solution}
Let us recall (see Proposition~\ref{prop:solution-RP}) that  for the case $c^L>c^R$ the solution to a Riemann problem contains only one $c$-shock. We use the vanishing viscosity criterion and study the travelling wave solutions in two particular cases of \eqref{eq:main_system_dissipation}: 
\begin{itemize}
    \item a system with capillary term and non-equilibrium adsorption ($\varepsilon_d=0$ and $\kappa=\varepsilon_r/\varepsilon_c$ is fixed)
\begin{equation}\label{eq:main_system_smooth_cap_non-eq}
\begin{cases} 
s_t + f(s, c)_x = \varepsilon_c (A(s,c) s_x)_x, \\
(cs + \alpha)_t + (cf(s,c))_x  = \varepsilon_c (cA(s,c) s_x)_x, \\ 
\varepsilon_r \alpha_t = a(c) - \alpha;
\end{cases}
\end{equation}
    \item a system with capillary and diffusion terms ($\varepsilon_r=0$ and $\kappa=\varepsilon_d/\varepsilon_c$ is fixed)
\begin{equation}\label{eq:main_system_smooth_cap_diff}
\begin{cases} 
s_t + f(s, c)_x = \varepsilon_c (A(s,c) s_x)_x, \\
(cs + a(c))_t + (cf(s,c))_x  = \varepsilon_c (cA(s,c) s_x)_x + \varepsilon_d c_{xx}.
\end{cases}
\end{equation}   
\end{itemize}

 Below we will consider the system \eqref{eq:main_system_smooth_cap_non-eq}. However, our proofs work with minimal corrections for the system \eqref{eq:main_system_smooth_cap_diff}.

\begin{definition}
We call a $c$-shock \textbf{admissible} if it could be obtained as a limit of smooth travelling wave solutions of \eqref{eq:main_system_smooth_cap_non-eq} as $\varepsilon_{c,r} \to 0$. 
\end{definition}

The following Theorem is the main result of the paper:
\begin{theorem}\label{Theorem1}
Consider a system of conservation laws~\eqref{eq:main_system} and the dissipative system~\eqref{eq:main_system_smooth_cap_non-eq} under assumptions (F1)--(F4) and (A1)--(A3).
There exist $0<v_{\min}<v_{\max}<\infty$, such that  for every $\kappa=\varepsilon_r/\varepsilon_c\in(0, +\infty)$, there exist unique 
\begin{itemize}
    \item points $s^-(\kappa)\in[0,1]$ and $s^+(\kappa)\in[0,1]$;
    \item velocity $v(\kappa)\in[v_{\min},v_{\max}]$,
\end{itemize}
such that the $c$-shock wave, connecting $u^-(\kappa)=(s^-(\kappa),1)$ and $u^+(\kappa)=(s^+(\kappa),0)$ with velocity $v(\kappa)$, is admissible and compatible by speeds in a sequence of waves \eqref{eq:solution-RP-1}. Moreover, $v$ is monotone in $\kappa$ and continuous; $v(\kappa)\to v_{\min}$ as $\kappa \to \infty$; $v(\kappa)\to v_{\max}$ as $\kappa \to 0$.
\end{theorem}
\begin{remark}
\label{rm:velocity-admiss}
    If a $c$-shock wave from $u^-$ to $u^+$ with velocity $v$ is compatible by speed in a sequence of waves~\eqref{eq:solution-RP-1} (with both $s$-waves present), then the following inequalities are satisfied:
    \begin{equation}
    \label{eq:compatibility-of-speeds}
        f_s(u^-) \leq v \leq f_s(u^+).
    \end{equation}
    Indeed, if the $s$-waves in a sequence~\eqref{eq:solution-RP-1} are smooth solutions (rarefaction waves), then~\eqref{eq:compatibility-of-speeds} is obtained by definition. If the $s$-waves are shock waves, then~\eqref{eq:compatibility-of-speeds} is a consequence of Oleinik admissibility condition for scalar conservation laws.
\end{remark}

\textbf{Example: ``boomerang'' model with vanishing viscosity criterion}

For every fixed $\kappa$ Theorem \ref{Theorem1} provides us with $v(\kappa)$, $u^-(\kappa)$ and $u^+(\kappa)$. Although the proof does not give an analytical expression for $v(\kappa)$, we can numerically calculate this function  and construct a solution to a Riemann problem for any given $\kappa$ (see Fig.~\ref{fig:BL_RS_vanvisc}abc). As $\kappa\to0$ we have $v(\kappa)\to v_{\max}$ and the solution tends to the solution obtained by Lax admissibility criterion.

\begin{figure}[H]
    \centering
    \includegraphics[width=0.31\textwidth]{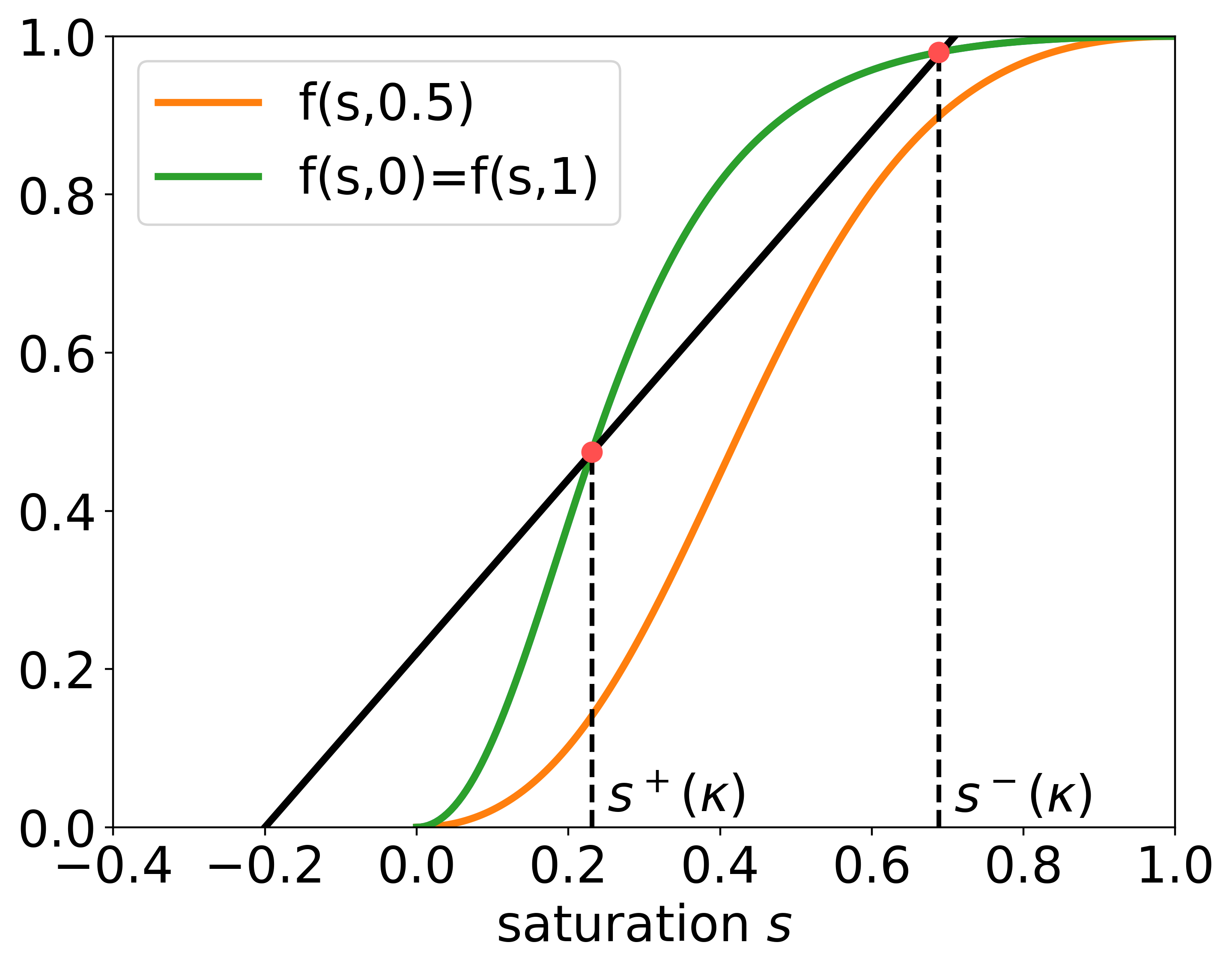}
   \includegraphics[width=0.31\textwidth]{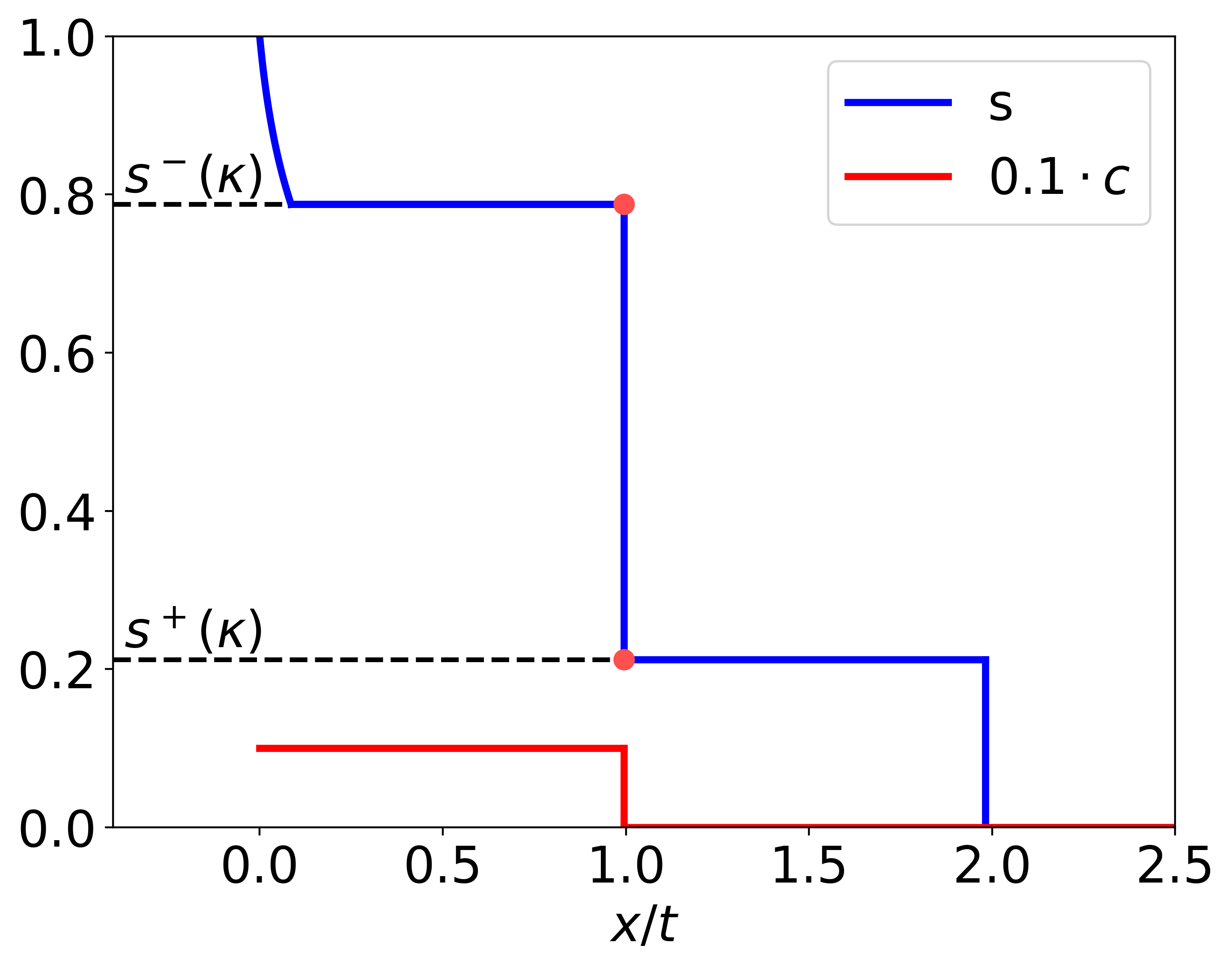}
   \includegraphics[width=0.327\textwidth]{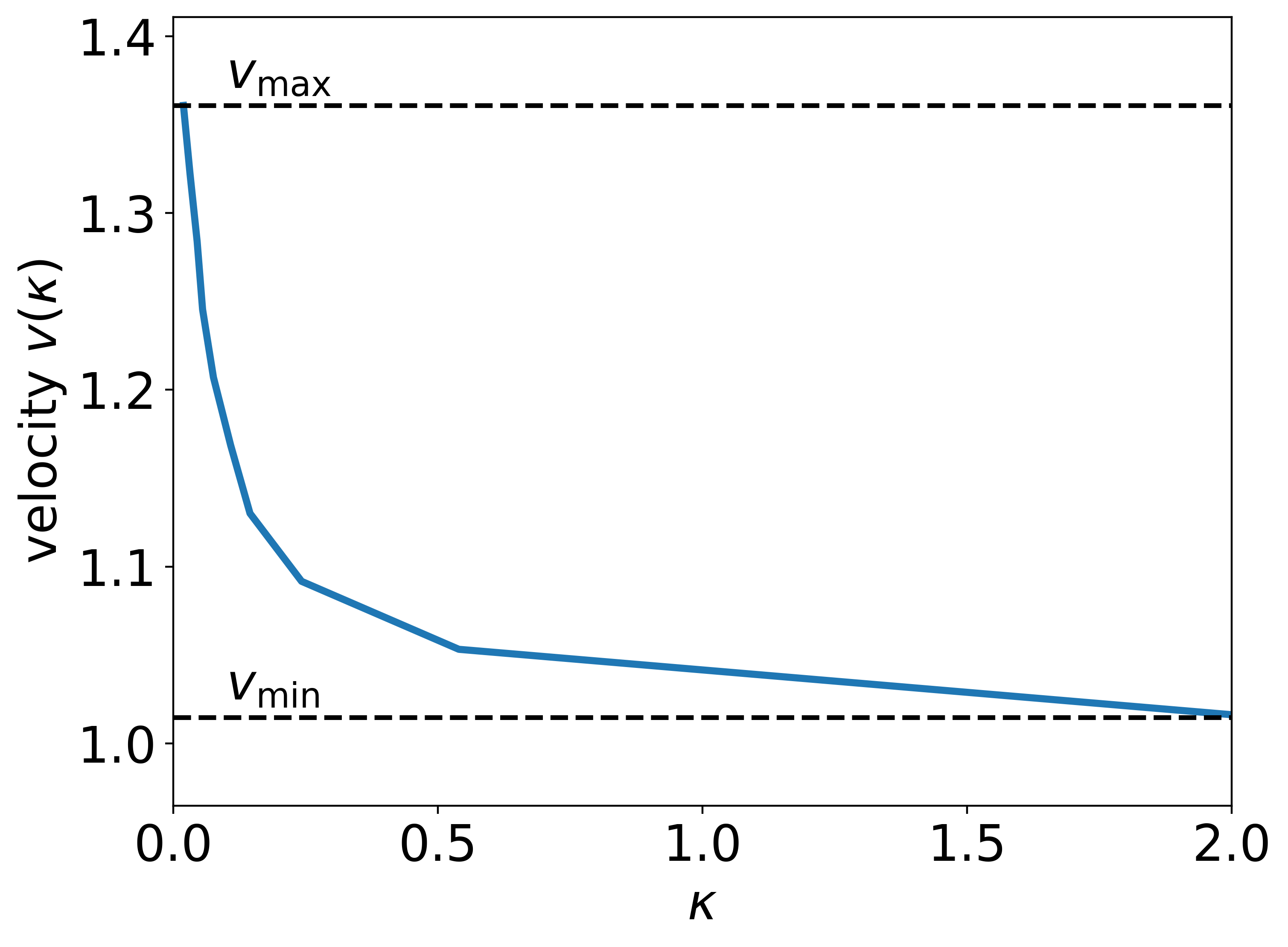}
    \\
    (a)\qquad\qquad\qquad\hfil 
    (b)\hfil\qquad\qquad\qquad 
    (c)\hfil
    \caption{``boomerang'' model: (a) $(s,f)$ diagram; 
    (b) solution to a Riemann problem for vanishing viscosity criterion  for some fixed $\kappa$; (c) graph of dependence $v$ on $\kappa$}
    \label{fig:BL_RS_vanvisc}
\end{figure}

\begin{figure}[H]
    \centering
    \qquad\qquad$\kappa=2$
    \hfill
    $\kappa=0.1$
    \hfill
    $\kappa=0.05$
    \hfill
    $\kappa=0$
    \qquad\qquad
    \qquad
    \hfill
    \\
    \includegraphics[width=0.244\textwidth]{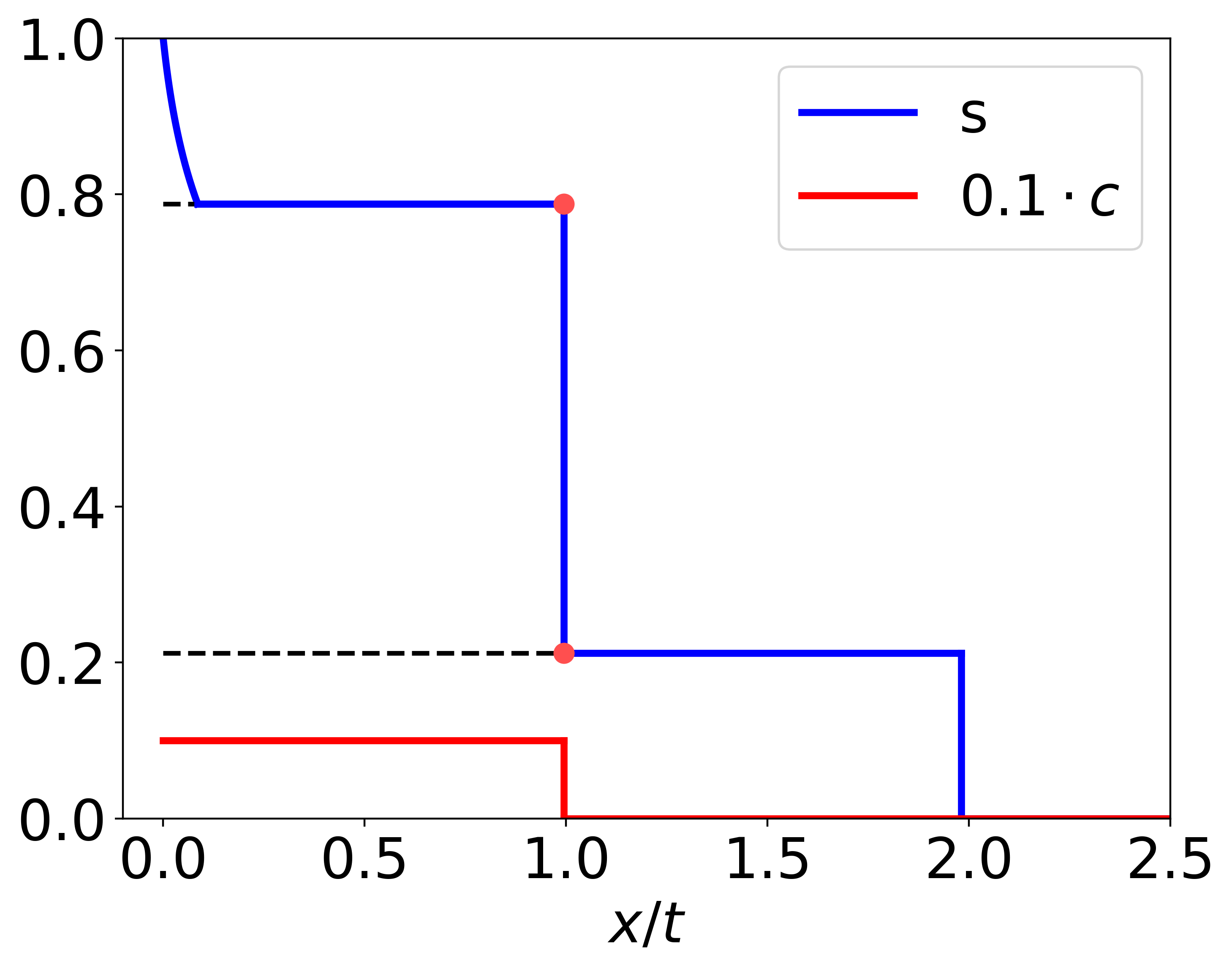}
   \includegraphics[width=0.244\textwidth]{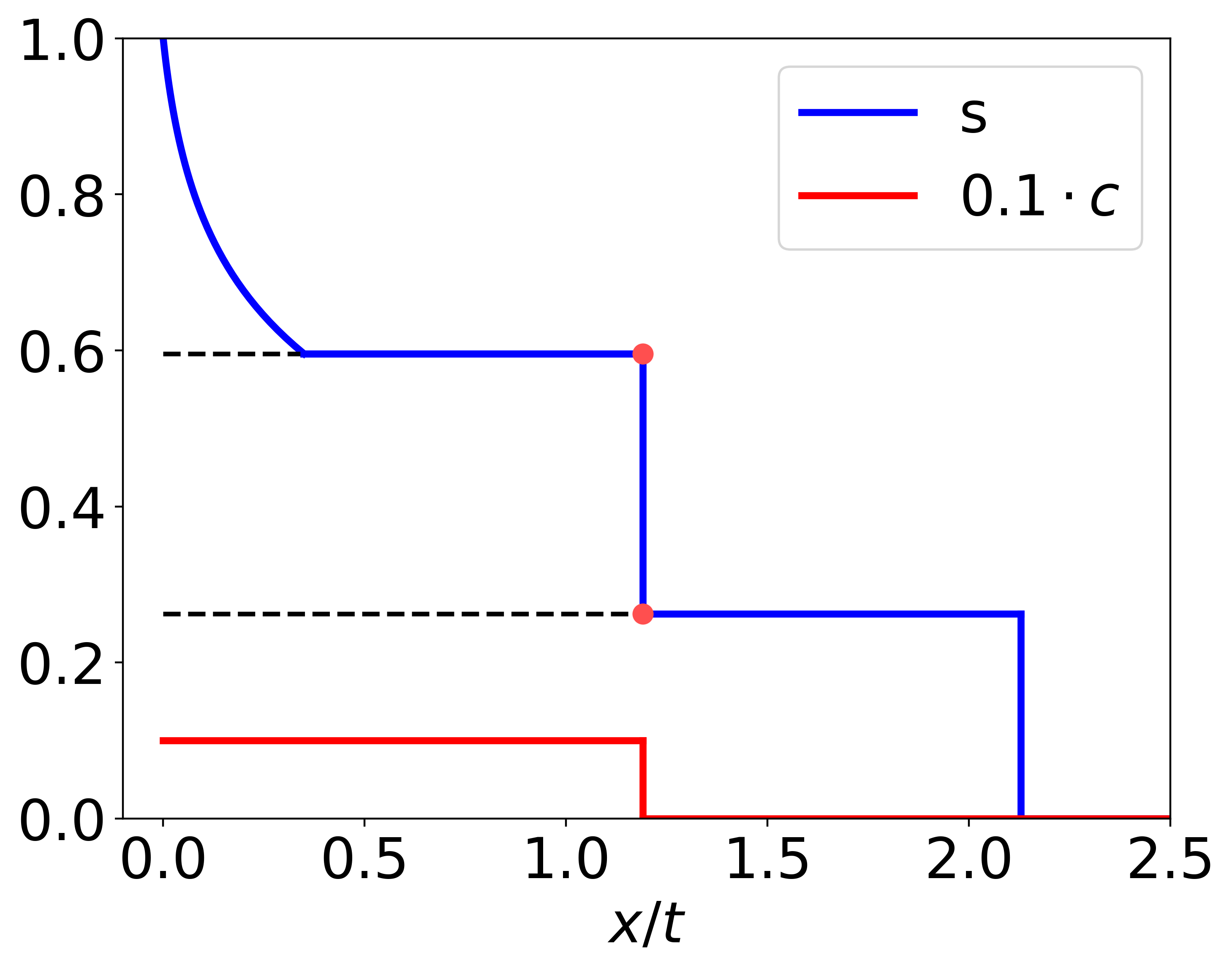}
   \includegraphics[width=0.244\textwidth]{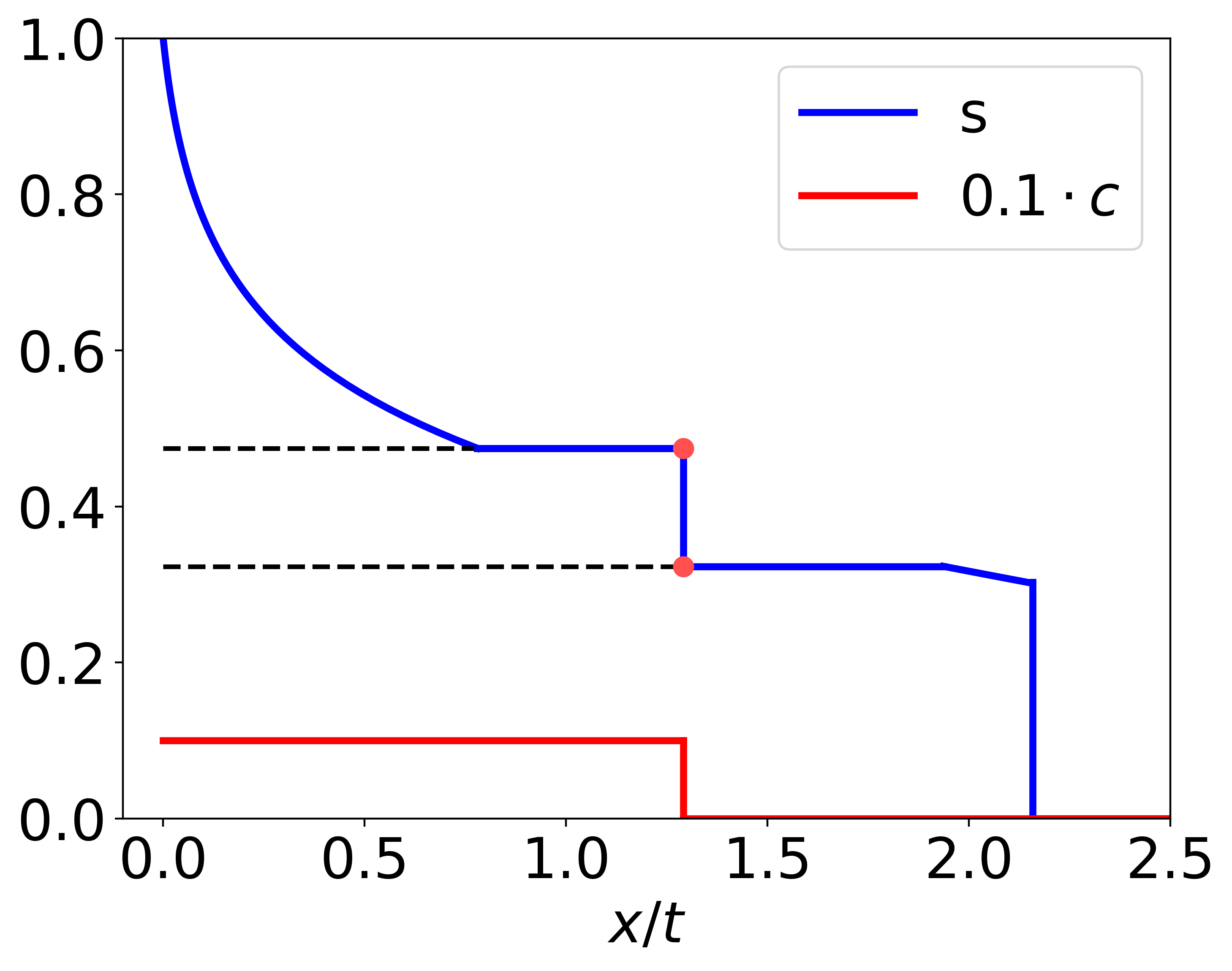}
   \includegraphics[width=0.244\textwidth]{RP_boomerang_Lax.png}
    \caption{solutions $(s,c)$ for different values of $\kappa\in\{2,0.1, 0.05, 0\}$}
    \label{fig:RP_boomerang_multiple}
\end{figure}
\section{Dynamical system for travelling wave solutions}
\label{sec:shock-waves}
\subsection{Dynamical system derivation}
\label{traveling_waves}

In this subsection we are looking for a travelling wave solution of the system \eqref{eq:main_system_smooth_cap_non-eq} connecting $(s^-, c^-)$ and $(s^+, c^+)$, i.e.
\begin{align*}
&\xi  = \varepsilon_c^{-1} (x - vt),
& &s(x,t) = s(\xi), & &c(x,t) = c(\xi), & & \alpha(x,t) = \alpha(\xi), \\
& & &s(\pm\infty) = s^\pm, & &c(\pm\infty)  = c^\pm, & &\alpha(\pm\infty) = a(c^\pm).
\end{align*}
In the above notation system \eqref{eq:main_system_smooth_cap_non-eq} can be rewritten as
\begin{equation*}
\begin{cases} 
-v s_\xi + f(s, c)_\xi = (A(s,c) s_\xi)_\xi, \\
-v (cs + \alpha)_\xi + (cf(s,c))_\xi  = (cA(s,c) s_\xi)_\xi, \\ 
-v \kappa \alpha_\xi = a(c) - \alpha,
\end{cases}
\end{equation*}
where $\kappa = \varepsilon_r/\varepsilon_c$. Integrating the first and the second equation we obtain
\begin{equation}\label{eq:almost_dyn_sys1}
\begin{cases} 
-v s + f(s, c) = A(s,c) s_\xi + vd_1, \\
-v cs - v \alpha + cf(s,c)  = cA(s,c) s_\xi + vd_2, \\ 
-v \kappa \alpha_\xi = a(c) - \alpha.
\end{cases}
\end{equation}
The values of $d_1$ and $d_2$ are obtained from the boundary conditions:
\begin{align*}
    vd_1 & = -vs^\pm + f(s^\pm, c^\pm), \\
    vd_2 & = v d_1 c^\pm - v a(c^\pm),
\end{align*}
namely
\begin{equation*}
d_1 = \dfrac{a(c^-) - a(c^+)}{c^- - c^+}, \quad d_2 =  \dfrac{c^+ a(c^-) - c^- a(c^+)}{c^- - c^+}.
\end{equation*}
Additionally, this form of boundary conditions yields us the value of $v$ and the Rankine-Hugoniot conditions \eqref{eq:RH-1}.
Multiplying the first equation of \eqref{eq:almost_dyn_sys1} by $c$, and subtracting this from the second equation in \eqref{eq:almost_dyn_sys1}, we obtain
\[
\alpha = d_1 c - d_2,
\]
thus $\alpha$ could be excluded from the system and we obtain the dynamical system
\begin{equation}\label{eq:dyn_sys_cap_non-eq}
\begin{cases} 
A(s,c) s_\xi = f(s, c) - v (s + d_1), \\
\kappa c_\xi = (vd_1)^{-1}(d_1 c - d_2 - a(c)).
\end{cases}
\end{equation}

\begin{remark}
Note that $d_1$ and $d_2$ do not depend on $v$, $\kappa$. 
Below we consider $v$ and $\kappa$ as parameters of the dynamical system~\eqref{eq:dyn_sys_cap_non-eq} and we search for pairs $(v, \kappa)$ that allow for a travelling wave solution. 
\end{remark}

\begin{remark}
The same transformations of the system~\eqref{eq:main_system_smooth_cap_diff} give us the dynamical system:
\begin{equation}\label{eq:dyn_sys_cap_diff}
\begin{cases} 
A(s,c) s_\xi = f(s, c) - v (s + d_1), \\
\kappa c_\xi = v (d_1 c - d_2 - a(c)),
\end{cases}
\end{equation}
where $\kappa = \varepsilon_d/\varepsilon_c$, which is very similar to \eqref{eq:dyn_sys_cap_non-eq}, so our results translate to it with minimal modifications in the proofs.  See Remark \ref{remark_difference_between_systems} for a better understanding why the change in monotonicity with respect to $v$ in the second equation does not translate into any significant change in proofs.
\end{remark}


\subsection{Phase portrait analysis}
\subsubsection{Phase portrait definition}

One of the instruments we utilize in the study of the dynamical system \eqref{eq:dyn_sys_cap_non-eq}
is the analysis of the phase portrait of the system. We draw phase portraits as vector fields $(s_\xi, c_\xi)$ on $\Omega = (0, 1)\times (c^+, c^-)$.  The primary aim of this section is the classification of possible portraits in order to exclude certain bad types of phase portraits that are possible in more general cases (see Remark~\ref{remark_bad_portraits}), and later to reduce the problem to studying one particular type of phase portraits that we call Type II below. 


We pay special attention to the nullclines (zeroes of the right-hand side of the dynamical system), drawing $\{(s,c): s_\xi = 0\}$ as {\bf black} curves and $\{(s,c): c_\xi = 0\}$ as {\color{red}\bf red} curves. The following observations apply to these curves (see Fig. \ref{fig:phase_portrait_1}):
\begin{itemize}
    \item since $A$ is never zero, black curves coincide with $\{(s,c): f(s,c)=v(s+d_1)\}$;
    \item since $a$ is concave (property (A3) in Section \ref{subsec:restrictions}),
    the red curves are always two lines $\{(s,c): d_1c - d_2 - a(c) = 0\} = [0,1] \times \{ c^+, c^- \}$;
    \item since $f$ is $S$-shaped for all $c$ (property (F3) in Section \ref{subsec:restrictions}), the black curves contain at most two points for each fixed concentration $c$;
    \item since $f$ changes monotonicity at most one time for each $s$ (property (F4) in Section \ref{subsec:restrictions}), the black curves contain at most two points for each fixed $s$.
\end{itemize}

\begin{figure}[H]
    \centering
    \includegraphics[width=0.6\textwidth]{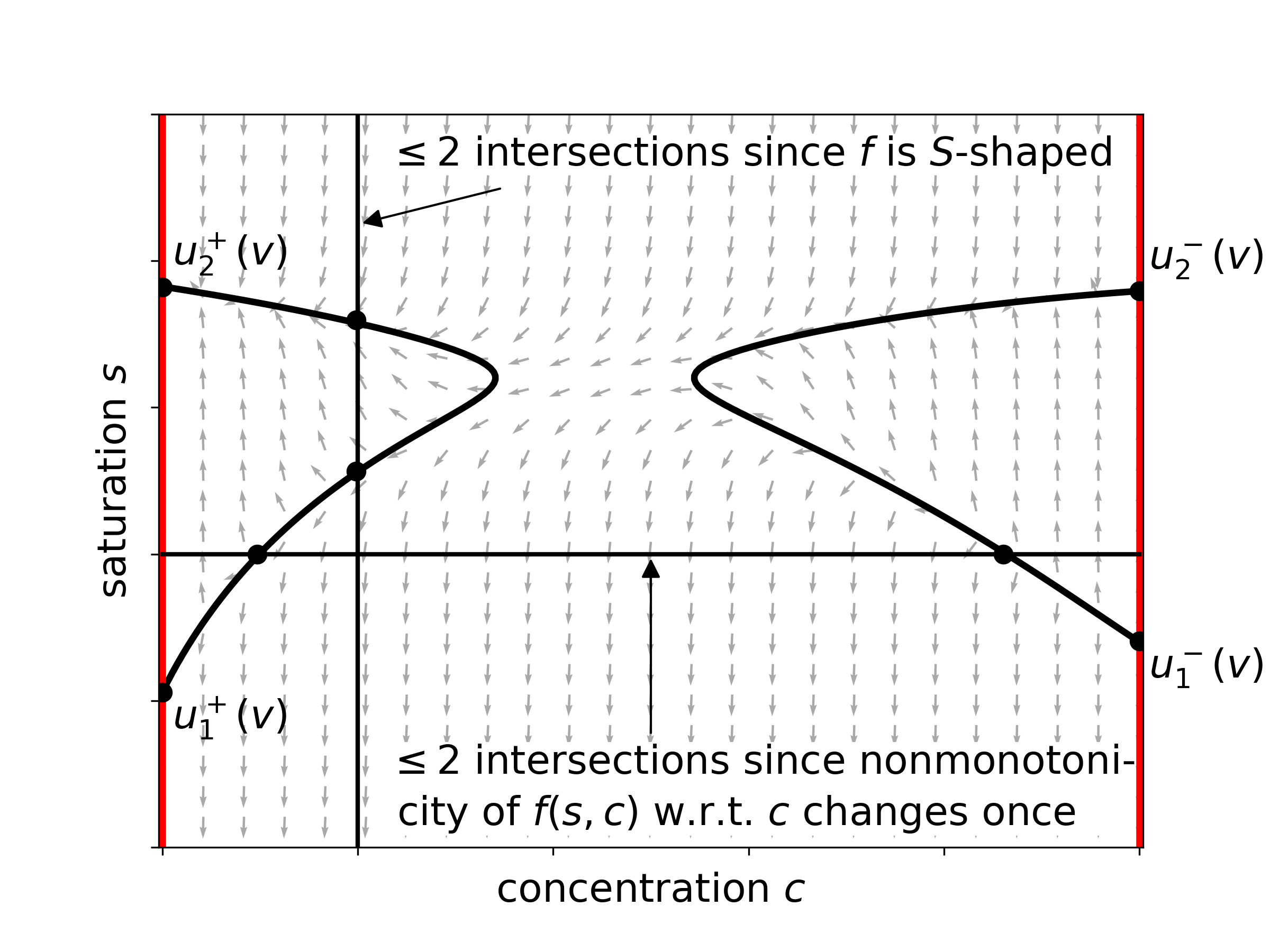}
    \caption{general properties of phase portraits}
    \label{fig:phase_portrait_1}
\end{figure}

\begin{remark}
\label{rm:critical_points}

On the intersections of black and red curves there are at most four fixed points of the dynamical system $u^\pm_{1,2}(v) := (s^\pm_{1,2}(v), c^\pm)$, where $s^\pm_1(v) \leq s^\pm_2(v)$. These critical points are in fact the only points that satisfy the Rankine-Hugoniot condition for the chosen value of $v$.

 To distinguish the type of critical points $u^\pm_{1,2}(v)$, we need to look at the eigenvalues of the Jacobian matrix
\[
\begin{bmatrix}
f_s(s,c) - v & f_c(s,c) \\
0 & h(d_1 - a'(c)) \\
\end{bmatrix},
\]
where $h=(vd_1)^{-1}$ for the system \eqref{eq:dyn_sys_cap_non-eq} and $h=v$ for the system \eqref{eq:dyn_sys_cap_diff}. When $s^\pm_1(v) < s^\pm_2(v)$ it is evident that
\begin{itemize}
    \item $f_s(u^-_1(v)) > v$, $d_1 - a'(c^-) > 0$ gives a source point at $u^-_1(v)$;
    \item $f_s(u^-_2(v)) < v$, $d_1 - a'(c^-) > 0$ gives a saddle point at $u^-_2(v)$;
    \item $f_s(u^+_1(v)) > v$, $d_1 - a'(c^+) < 0$ gives a saddle point at $u^+_1(v)$;
    \item $f_s(u^+_2(v)) < v$, $d_1 - a'(c^+) < 0$ gives a sink point at $u^+_2(v)$.
\end{itemize}
When $s^+_1(v) = s^+_2(v)$, we have $f'_s(u^{+}_{1,2}(v)) = v $, which gives a saddle-node point at  $u^{+}_{1,2}(v)$. Similarly, when $s^-_1(v) = s^-_2(v)$, we have a saddle-node at $u^{-}_{1,2}(v)$.
\end{remark}

\begin{remark}
Note that the nullclines are the same for the dynamical systems \eqref{eq:dyn_sys_cap_non-eq}, \eqref{eq:dyn_sys_cap_diff}, thus the classification below applies to both.

\end{remark}

\subsubsection{Phase portrait classification}


There are five wide classes of phase portraits (see Fig. \ref{fig:phase_portrait_all}): 

\begin{itemize}
    \item Type 0. Black curves contain exactly one point for each concentration $c \in (c^+, c^-)$ and thus there is a curve connecting the red lines from $u^+_1(v)$ to $u^-_1(v)$.
    \item Type I. Black curves contain exactly two points for each concentration $c \in (c^+, c^-)$ and thus there are two non-intersecting curves connecting the red lines: one from $u^+_1(v)$ to $u^-_1(v)$ and another from $u^+_2(v)$ to $u^-_2(v)$.
    \item Type II. There exists an interval $(c_1, c_2)$, on which there are no points of black curves present, but for $c^+$ and $c^-$ there are two critical points; also, $s^+_1(v) < s^-_2(v)$. Thus, black curves split into separate left and right branches and the right branch is not fully lower than the left one.
    \item Type III. There are two branches as in Type II, but the right one is fully lower than the left one, i.e. $s^+_1(v) > s^-_2(v)$.
    \item Type IV. One of the red lines does not contain any critical points.
\end{itemize}
\begin{figure}[H]
    \centering
    \includegraphics[width=0.178\textwidth]{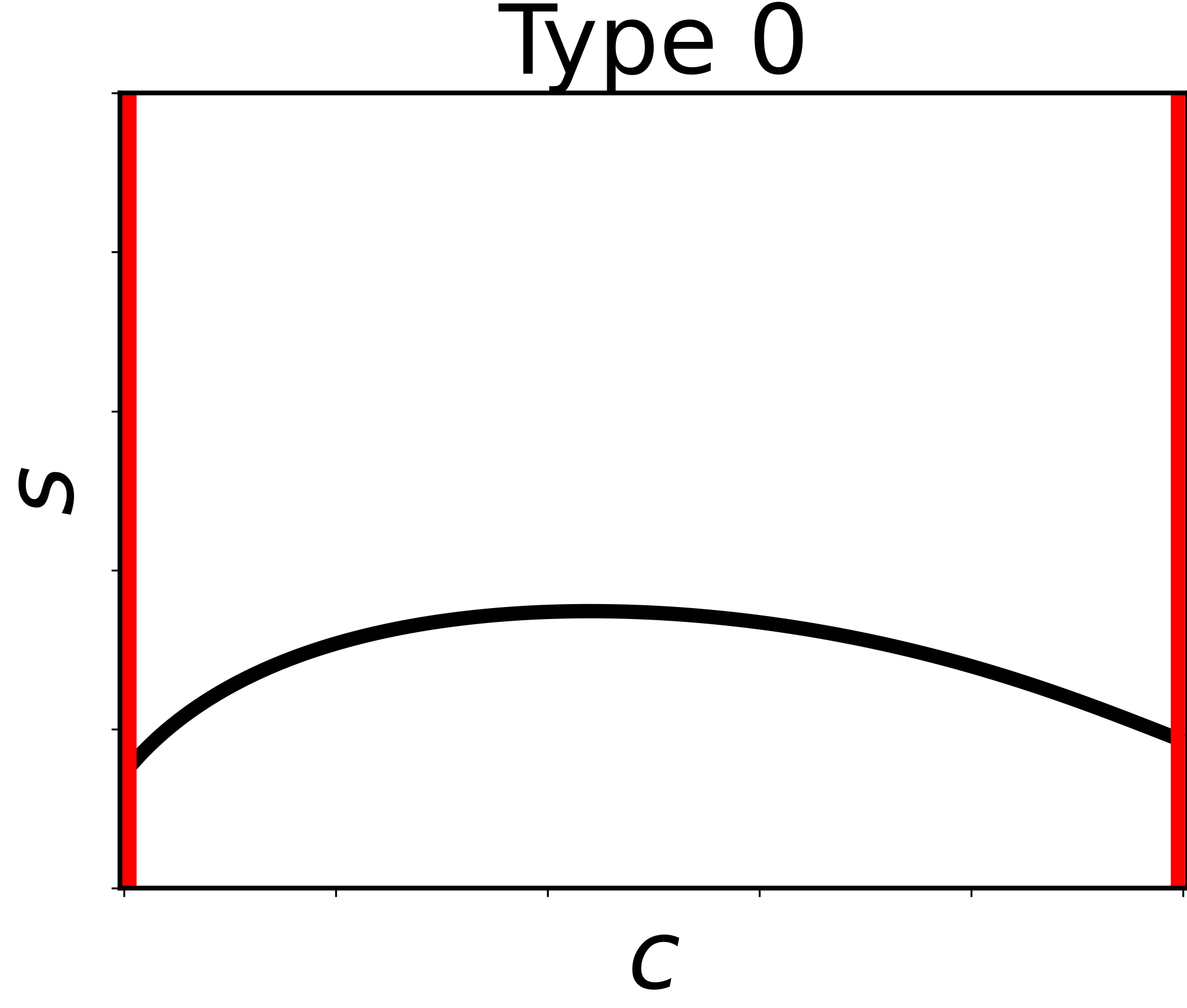}
    \hfil
    \includegraphics[width=0.16\textwidth]{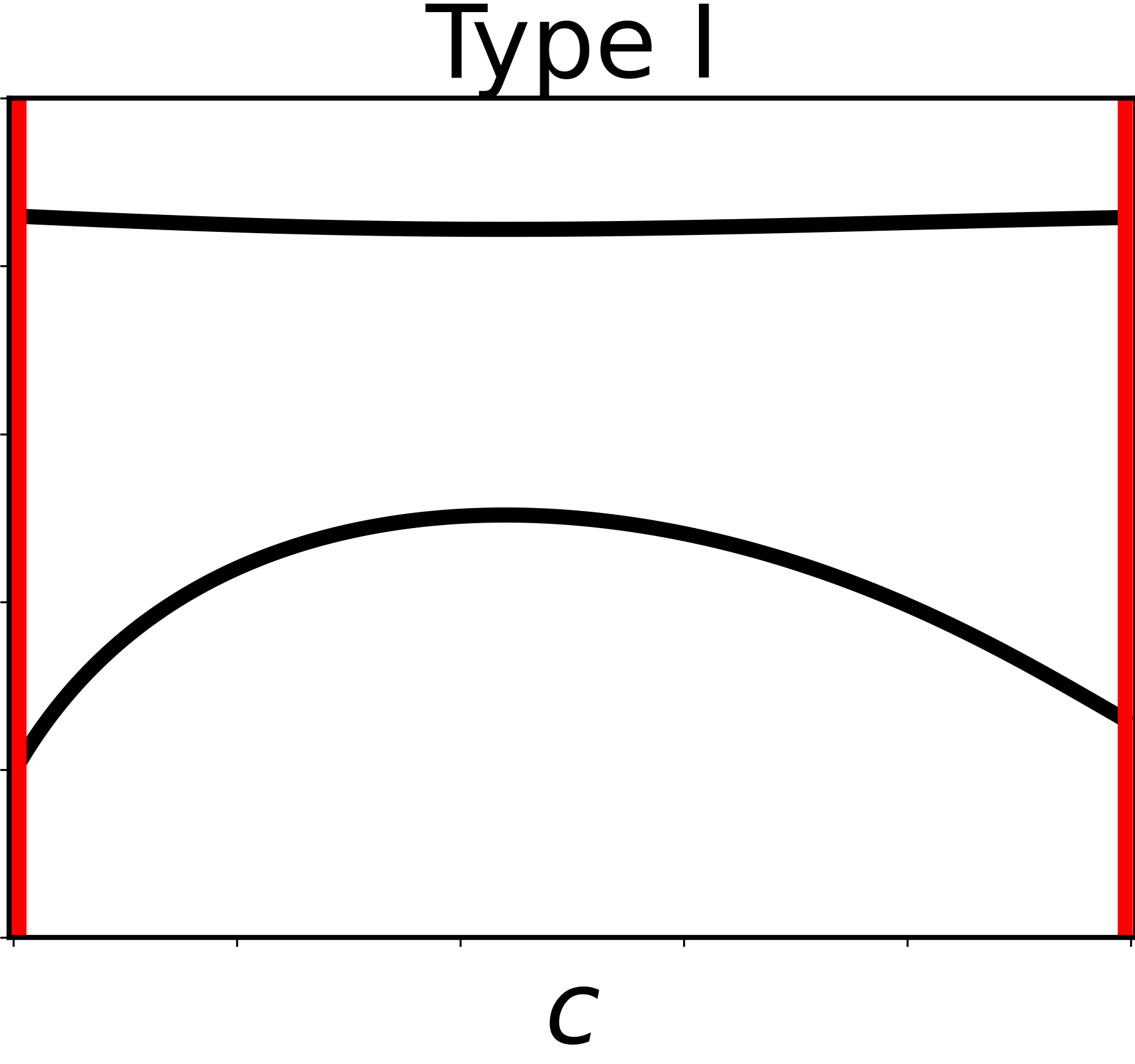}
    \hfil
    \includegraphics[width=0.16\textwidth]{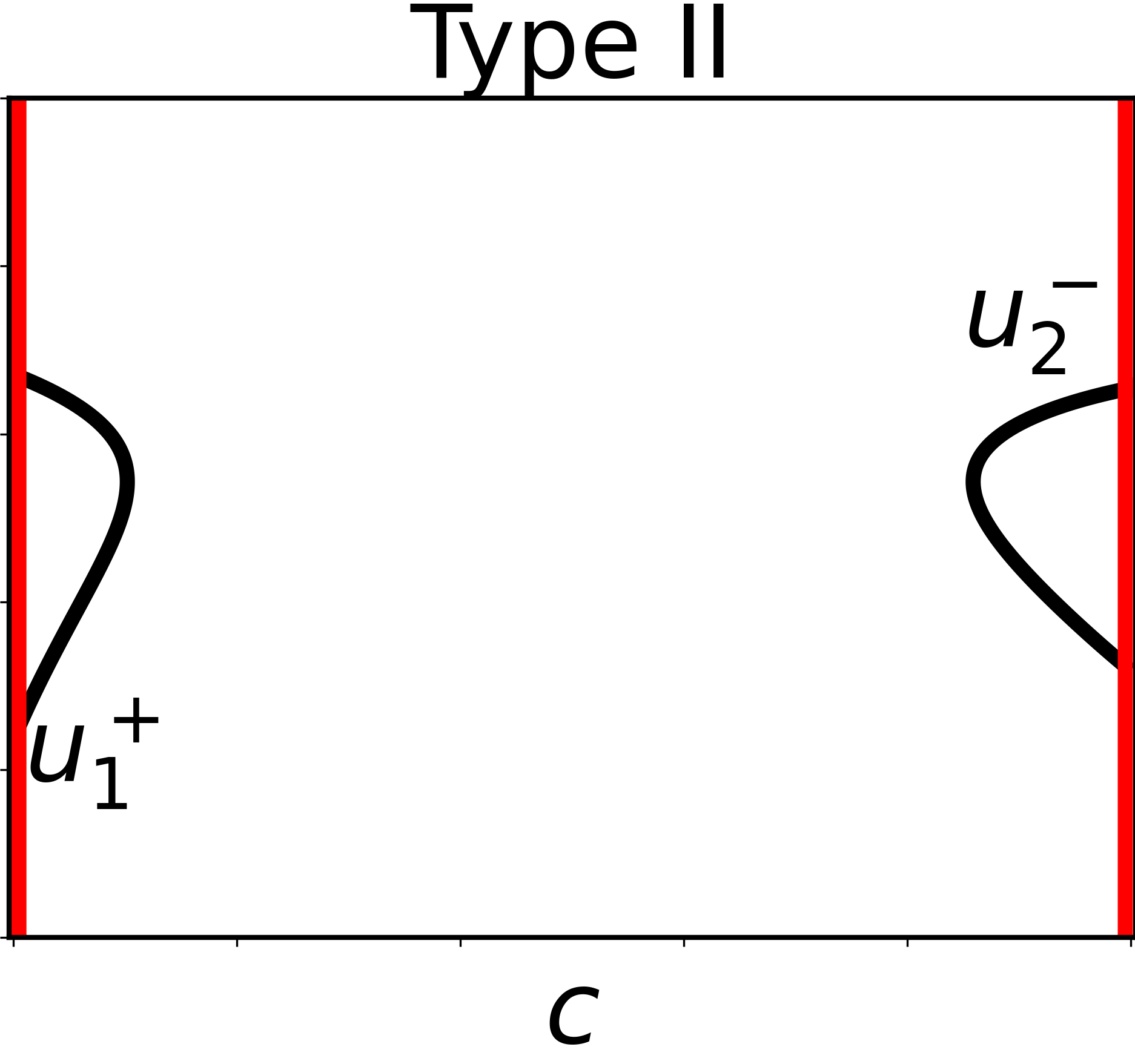}
    \hfil
    \includegraphics[width=0.16\textwidth]{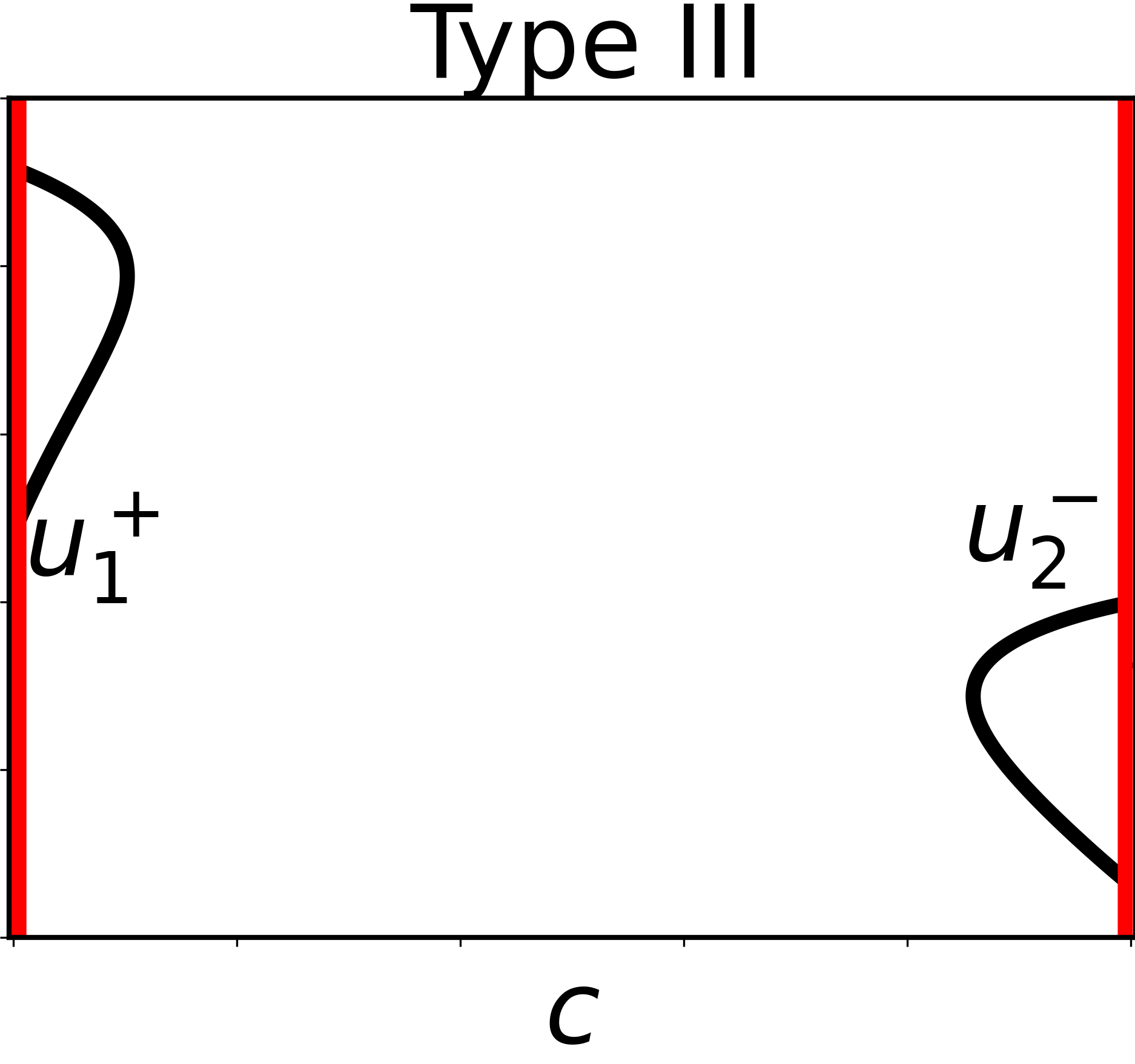}
    \hfil
    \includegraphics[width=0.16\textwidth]{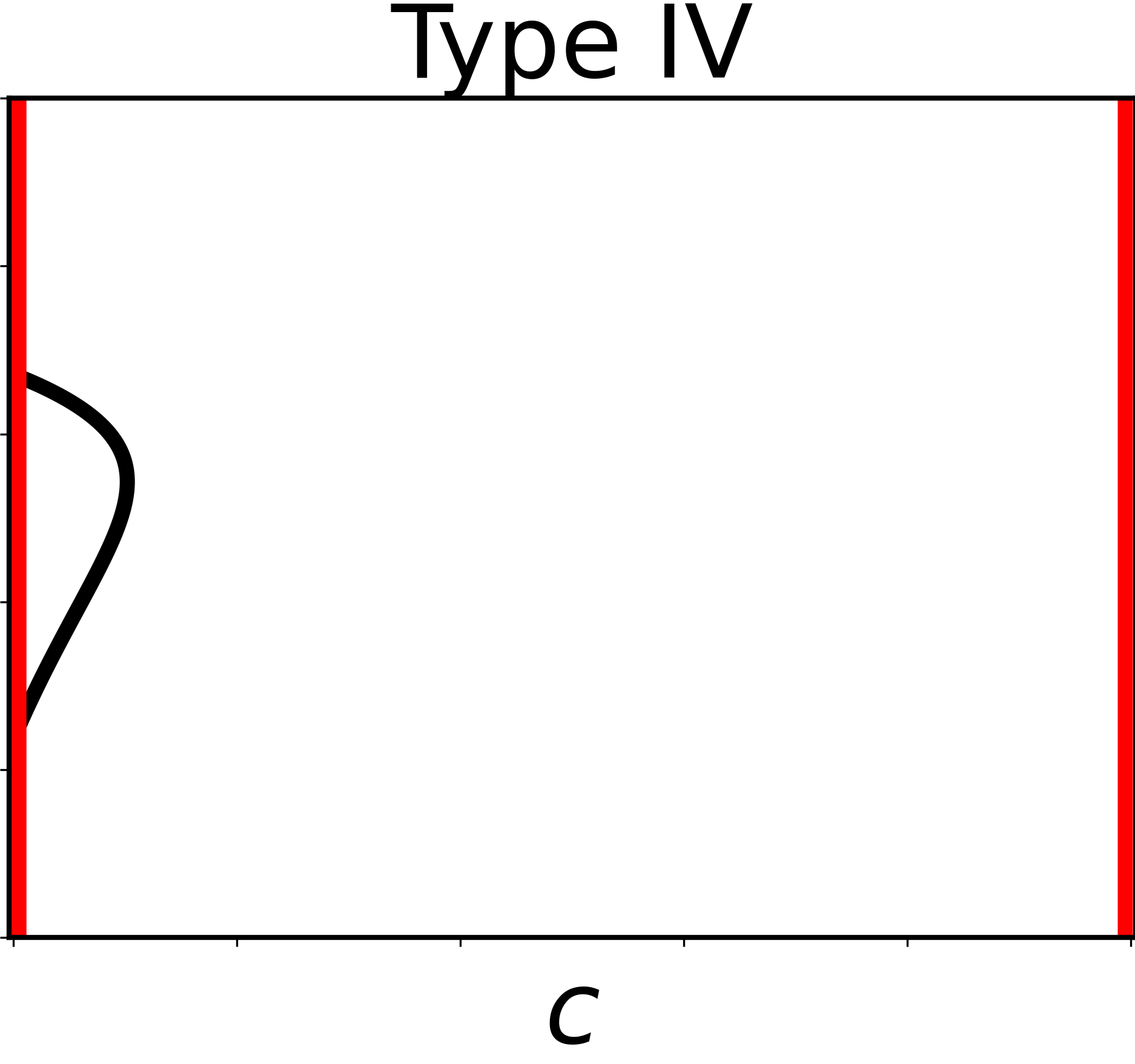}
    \caption{five wide classes of phase portraits}
    \label{fig:phase_portrait_all}
\end{figure}

In addition, there are several singular intermediate types of phase portraits that are essentially border cases for the wide classes described above (see Fig. \ref{fig:phase_portrait_border_type}): 

\begin{itemize}
    \item Type 0-I. Similar to Type I, but the upper black curve coincides with the border $s=1$.
    \item Type I-II. There exists one concentration $c_{*}\neq c^\pm$ for which black curves contain only one point. For all other concentrations there are two points on black curves.
    \item Type II-III. Only one value of $s$ has two points, other values have at most one point, i.e. $s^+_1(v) = s^-_2(v)$.
    \item Type II-IV. One of the branches of Type II portrait degenerates into a point. Thus $c^+$ or $c^-$ only have one critical point on them, i.e. $s^+_1(v) = s^+_2(v)$ or $s^-_1(v) = s^-_2(v)$. 
    \item Type III-IV. Similar to Type II-IV, but for Type III portrait instead of Type II.
\end{itemize}

\begin{figure}[H]
\centering
\includegraphics[width=0.178\textwidth]{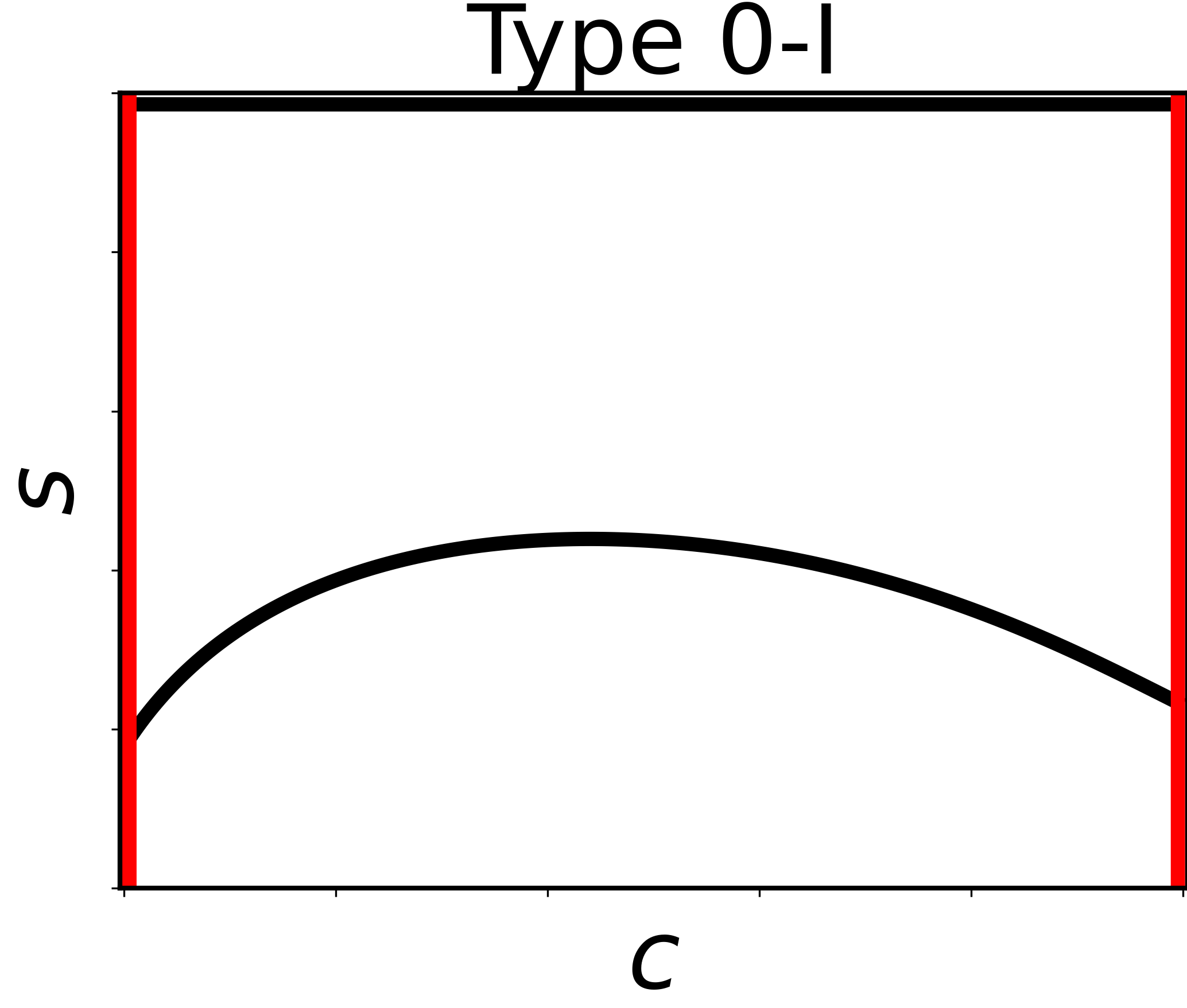}
\hfil
\includegraphics[width=0.16\textwidth]{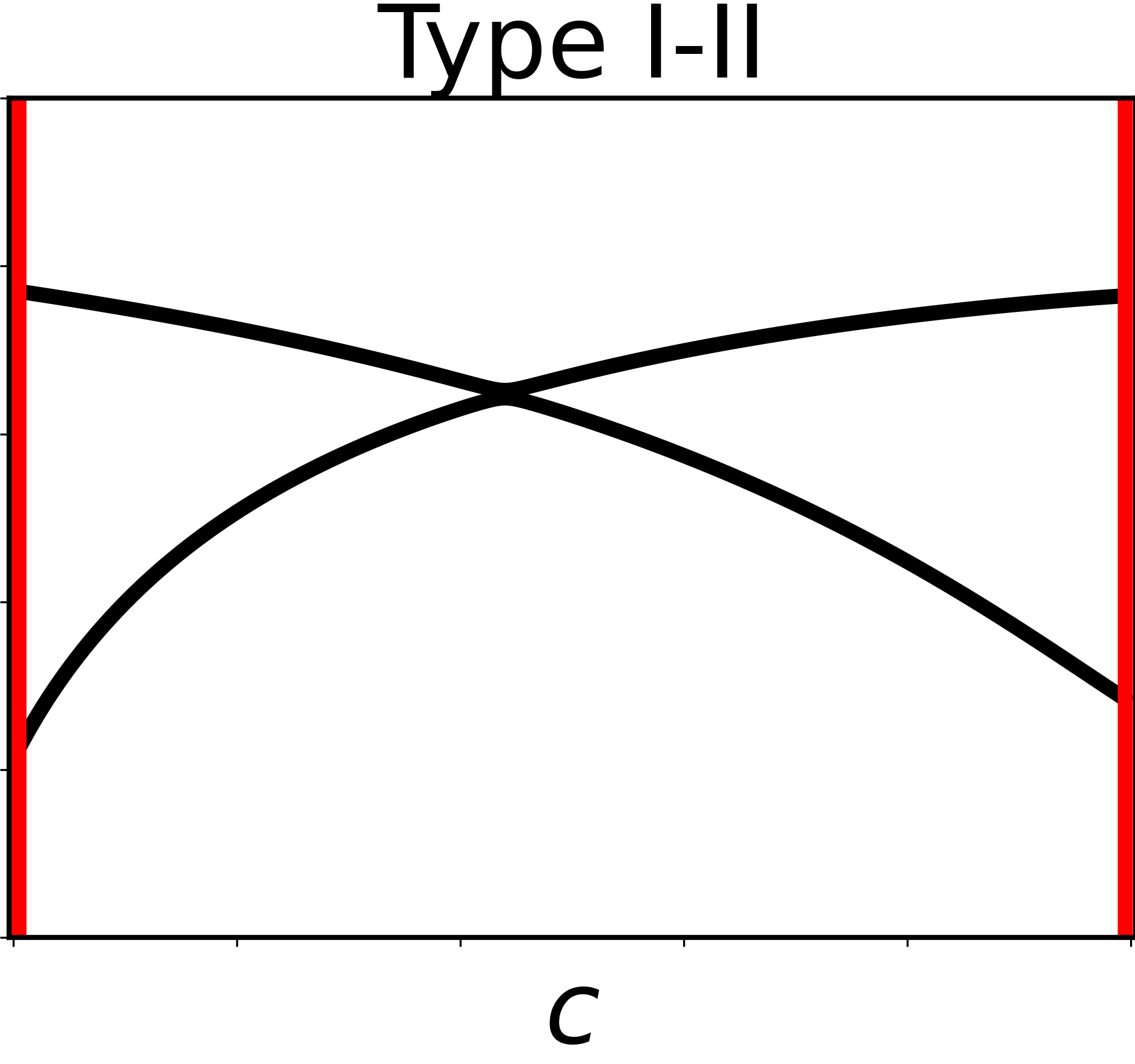}
\hfil
\includegraphics[width=0.16\textwidth]{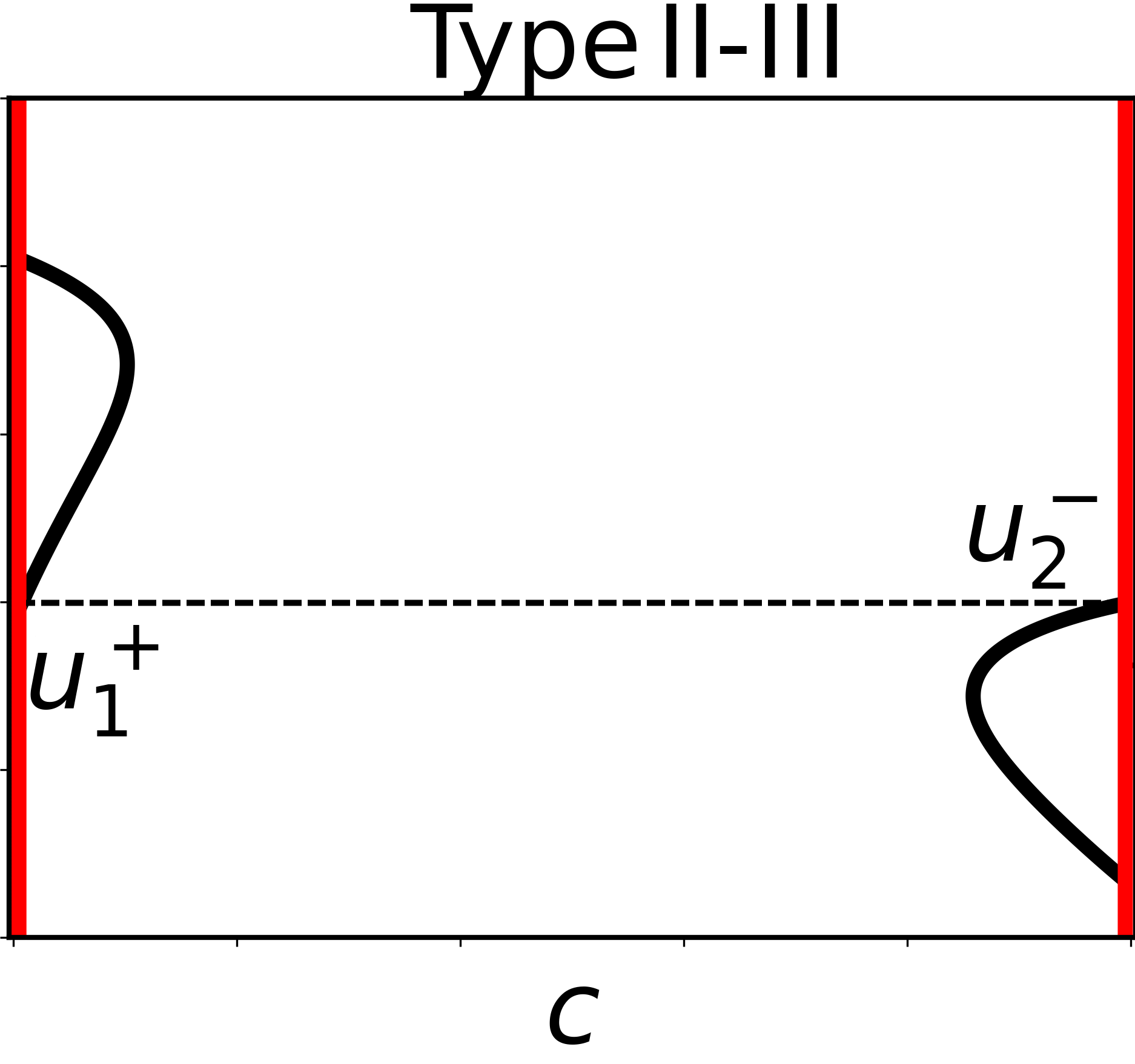}
\hfil
\includegraphics[width=0.16\textwidth]{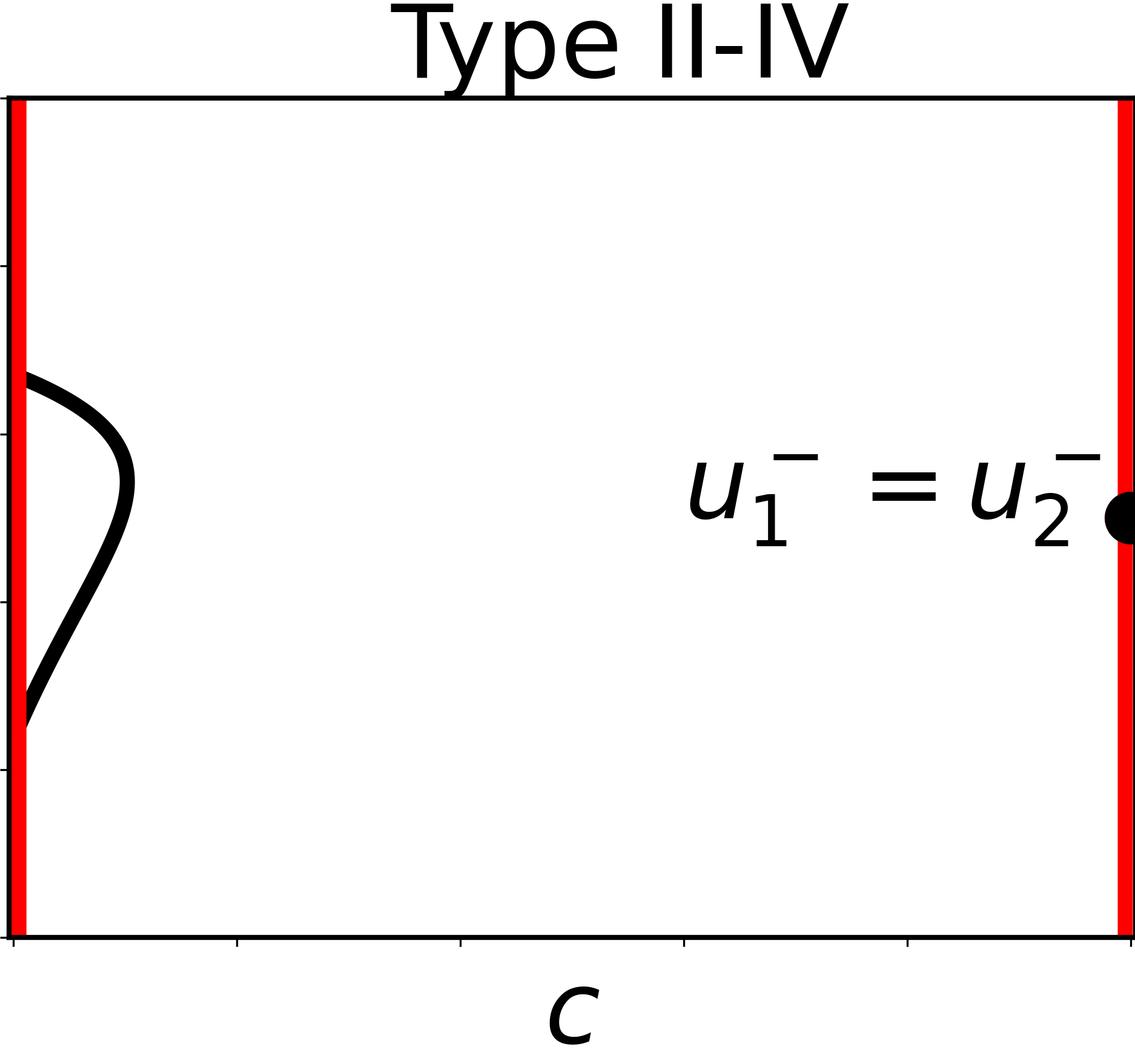}
\hfil
\includegraphics[width=0.16\textwidth]{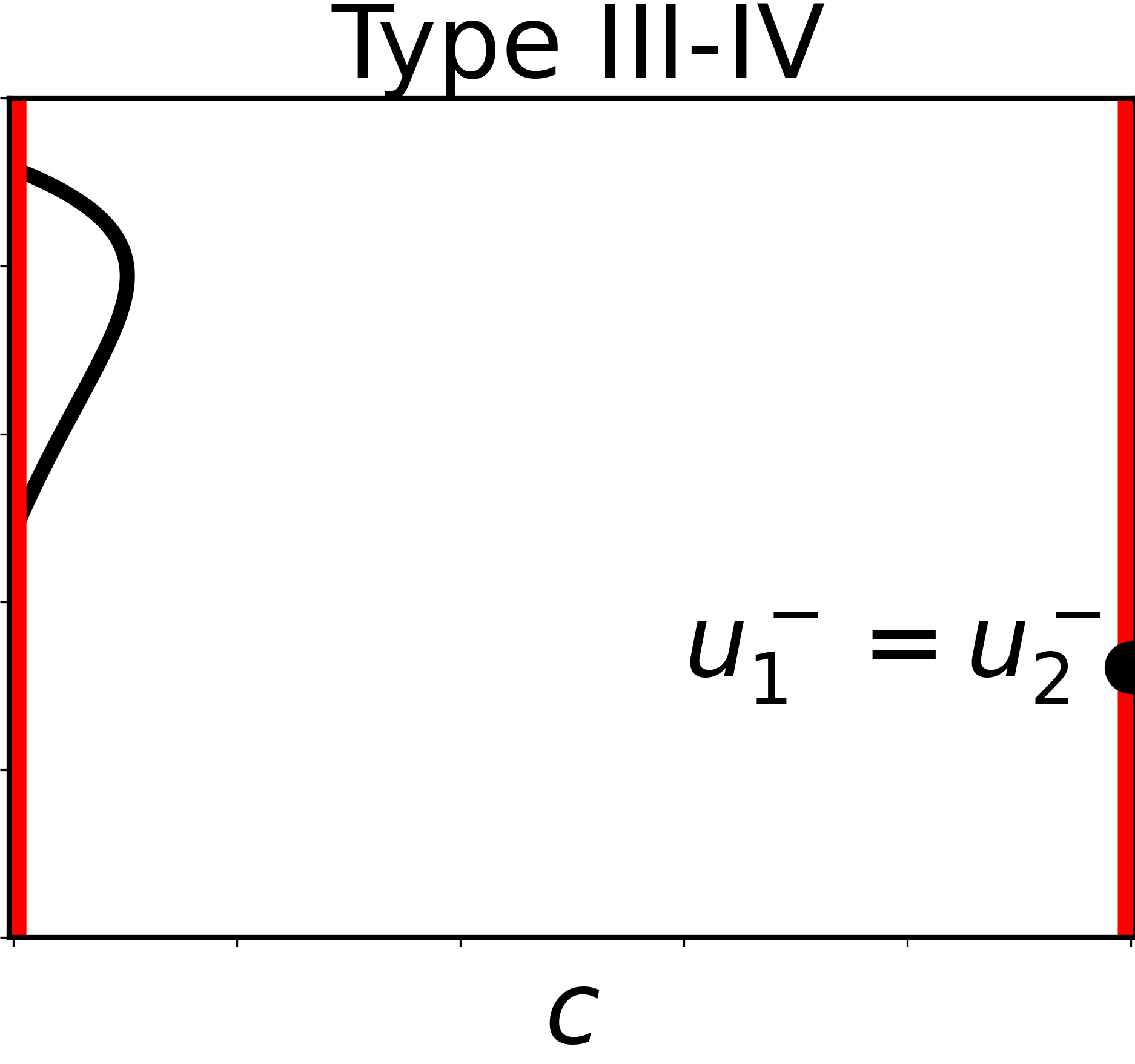}
    \caption{intermediate types of phase portraits, appearing under the assumptions of Section \ref{subsec:restrictions}}
    \label{fig:phase_portrait_border_type}
\end{figure}

Note that $s_\xi=A^{-1}(s,c)(f(s, c) - vs -vd_1)$ is monotone with respect to $v$ at each point $(s,c)$, so the phase portrait type evolves in a predictable manner as $v$ changes (see Fig. \ref{fig:phase_portrait_evolution}):
\begin{itemize}
    \item For $v$ close to zero  we always have Type~0 phase portrait.
    \item With increasing $v$ the type of the portrait changes in increasing order.
    \item Type~III portrait might be omitted.
    \item Intermediate types described above connect portraits corresponding to their numbers.
\end{itemize}
Each boundary type corresponds to a certain value of $v$.  We denote these velocities bounding the Type II by $v_{\min}$ and $v_{\max}$ (these are the same $v_{\min}$ and $v_{\max}$ as in Theorem~\ref{Theorem1}). It is possible to calculate these velocities from $f$: 
\begin{itemize}
    
    
    \item $v_{\min}$ is the velocity that gives a Type I-II portrait and is the minimum among the slopes of the tangent lines from the point $Q=(-d_1,0)$ to the family of functions $\{f(\cdot,c)\}_{c\in[0,1]}$. 
    
    \item $v_{\max}$ is the velocity of either Type II-III or Type II-IV, depending on whether Type III is omitted:
    \begin{itemize}
    \item $v_{\max}$ for Type II-III is the largest slope of the lines from point $Q$ to the intersection points of $f(s,c^-)$ and $f(s, c^+)$;
    \item $v_{\max}$  for Type II-IV (if we excluded Type II-III) is the smallest slope of the tangent lines from point $Q$ to  $f(s,c^-)$ and $f(s,c^+)$.
    \end{itemize}
    In order to determine if we have Type II-III, draw a line from point $Q$ through the intersection point $(s_2, f(s_2,c^+)) = (s_2, f(s_2, c^-))$ realising the highest angle (if there are no intersection points, then there is no Type III). If the line's other intersection $(s_1, f(s_1, c^-))$ with $f(s,c^-)$ is to the left of the curves' intersection ($s_1 < s_2$), and the line's other intersection $(s_3, f(s_3, c^+))$ with $f(s,c^+)$ is to the right of the curves' intersection ($s_2<s_3$), then we have Type II-III (see Fig. \ref{fig:Type23-vmax}). Otherwise Type II-III and Type III are omitted, and we have Type II-IV.
    \begin{figure}[H]
\centering
\includegraphics[height=0.25\textheight]{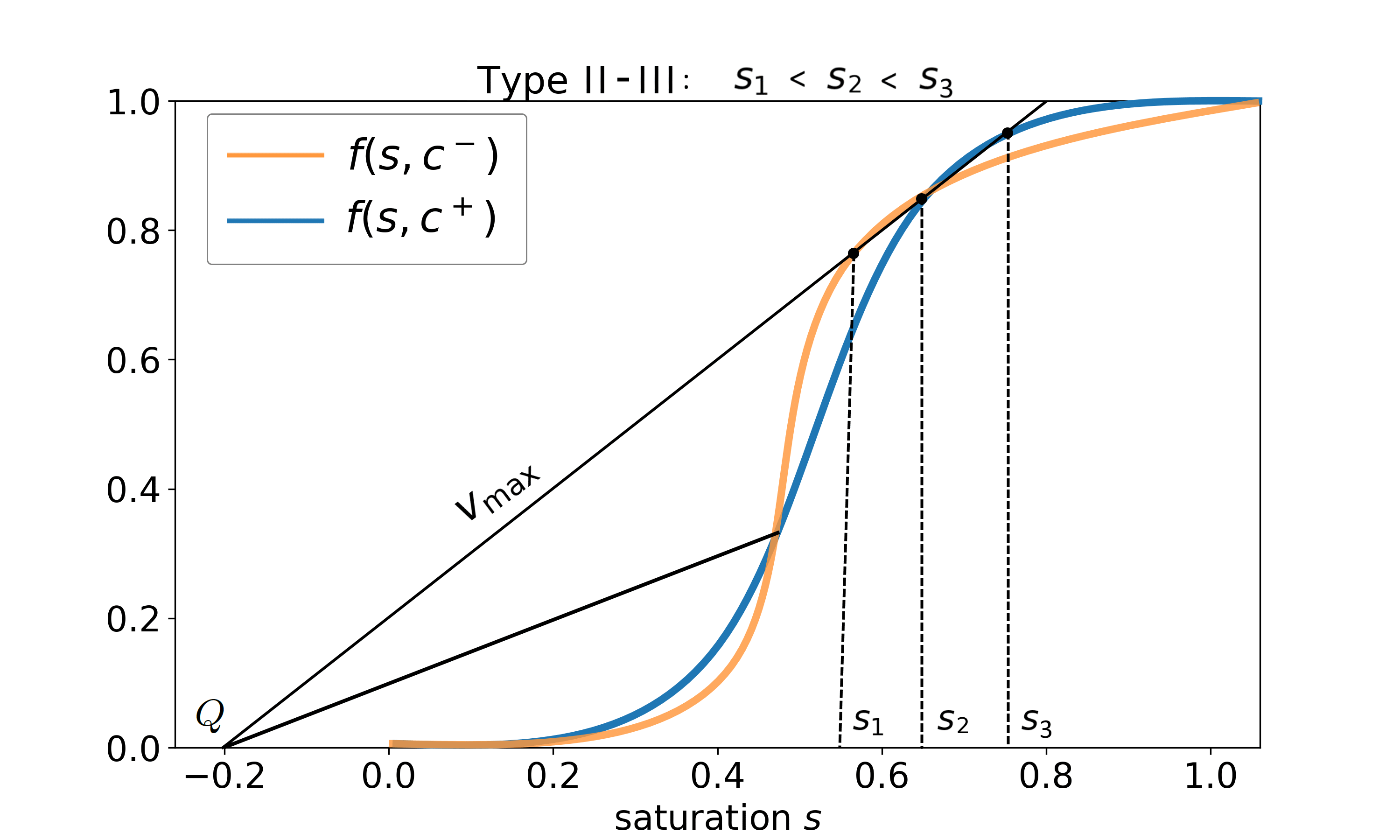}
    \caption{example of fractional flow functions corresponding for Type II-III phase portrait}
    \label{fig:Type23-vmax}
\end{figure}

\end{itemize}

\begin{figure}[H]
\centering
\includegraphics[height=0.25\textheight]{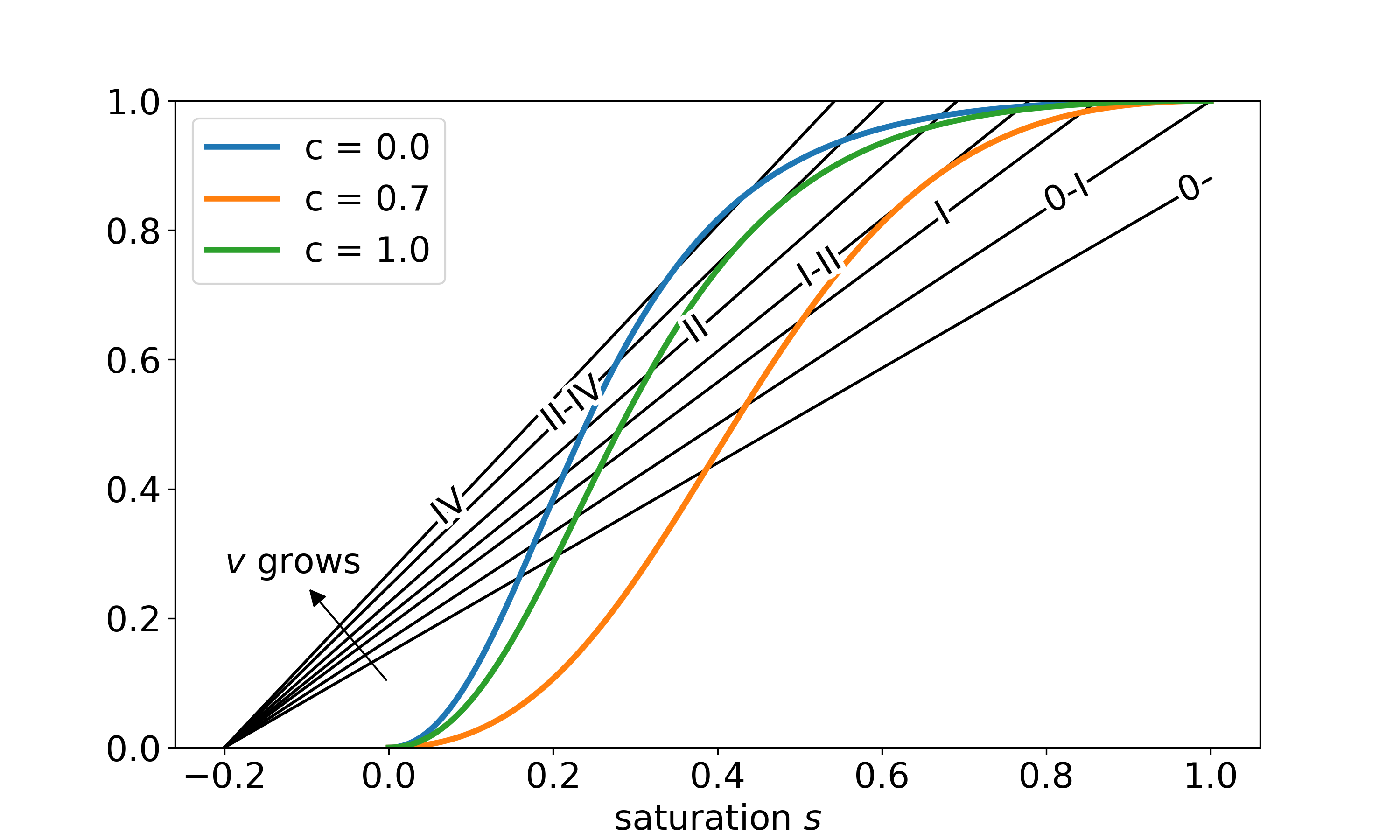}
\includegraphics[height=0.25\textheight]{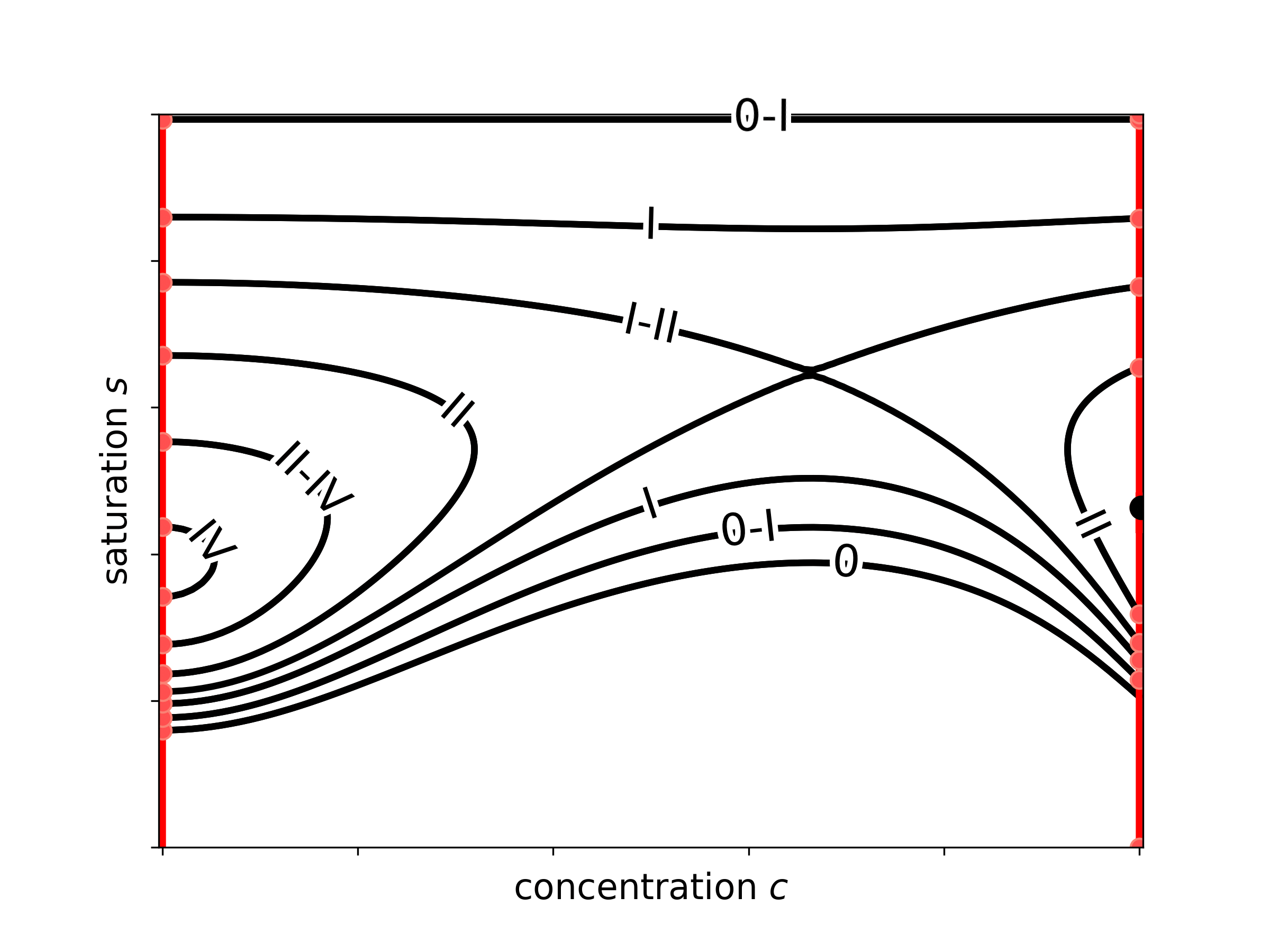}
    \caption{phase portrait evolution as $v$ grows: Type 0 $\to$ Type I $\to$ Type II $\to$ Type IV }
    \label{fig:phase_portrait_evolution}
\end{figure}

\begin{remark}\label{remark_bad_portraits}
Note, that transitions directly from Type~I to Type~III, from Type~I to Type~IV and from Type~0 to any type other than Type~I are not possible under our restrictions on function~$f$. Note as well, that the same restrictions (see Fig.~\ref{fig:phase_portrait_1}) eliminate the possibility of the following types of portraits (Fig.~\ref{fig:bad_phase_portrait}): 

\begin{figure}[H]
  \centering
    \includegraphics[width=0.48\textwidth]{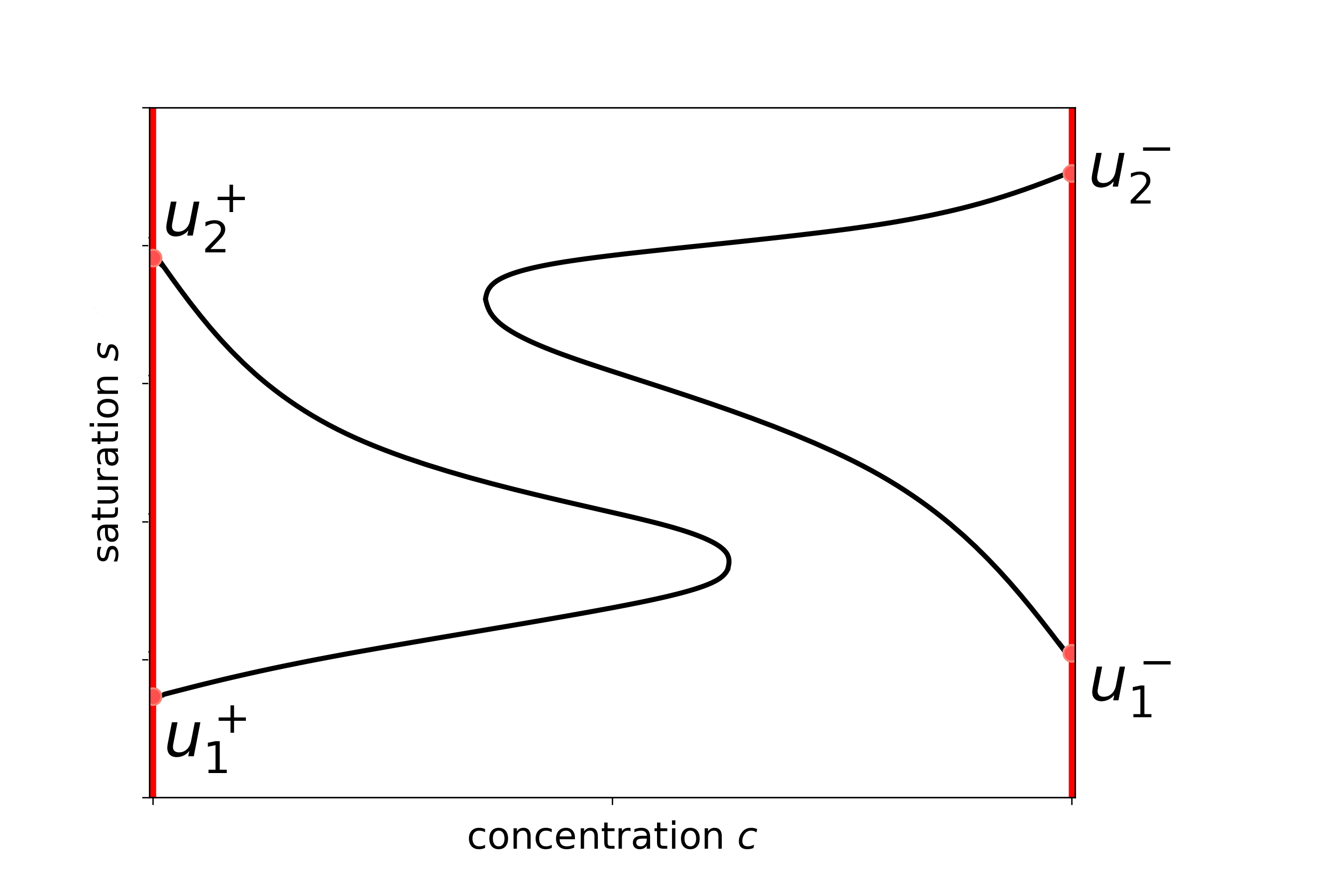}
    \includegraphics[width=0.48\textwidth]{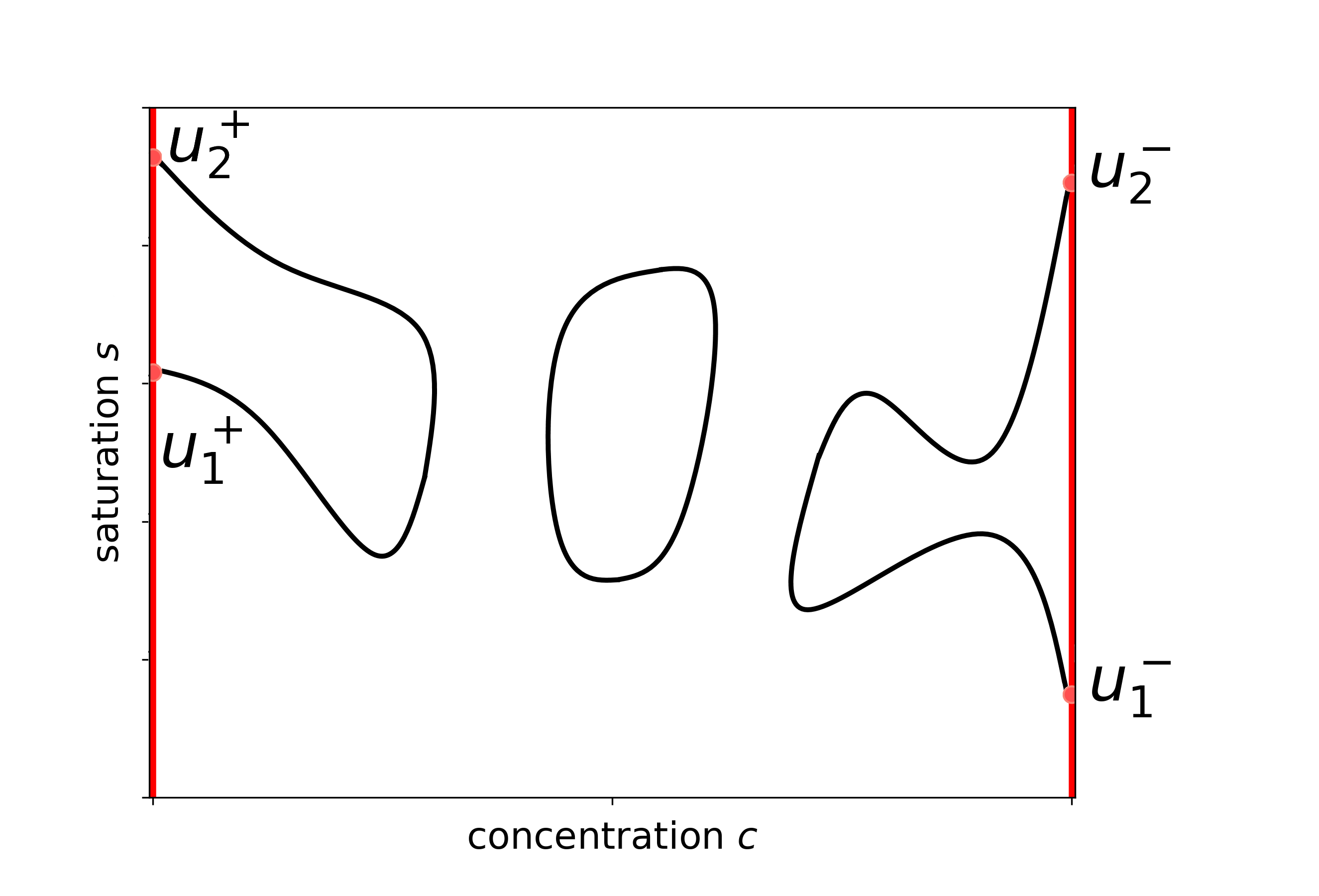}\\
    (a)\qquad\qquad\hfil \qquad\qquad\qquad(b)\hfil
    \caption{(a) an overlap preventing saddle points connection, possible if $f$ is not $S$-shaped; \\
    (b) an emergence of separated connectivity components, possible if non-monotonicity is more complex }
    \label{fig:bad_phase_portrait}
\end{figure}

Further generalizations (see Section~\ref{sec:generalisations} for further discussion on the topic) of this work would have to deal with some or all of this types.
\end{remark}

\subsection{General properties of trajectories}

In this subsection we formulate some basic properties of trajectories of the dynamical systems~\eqref{eq:dyn_sys_cap_non-eq} and \eqref{eq:dyn_sys_cap_diff}. 
The proofs of these properties are provided in the Appendix A.


\begin{proposition}\label{prop1}
Consider one of the dynamical systems \eqref{eq:dyn_sys_cap_non-eq}, \eqref{eq:dyn_sys_cap_diff}. Then:
\begin{enumerate}[label=\Alph*)]
    \item For all $v, \kappa > 0$ the solution with initial values in $\Omega$ exists, is $C^1$ smooth and depends continuously on $v, \kappa$ and initial values on any compact inside $\Omega$.
    \item Solutions form non-intersecting $C^1$ smooth orbits in $\Omega$.
    \item Every trajectory in $\Omega$:
    \begin{itemize}
        \item begins either in  $u^-_{1,2}(v)$ or on the border $\{s=1\}$;
        \item goes from right to left ($c$ always decreases);
        \item ends either in $u^+_{1,2}(v)$ or on the border $\{s=0\}$.
    \end{itemize}
    \item Any orbit can be represented as a graph of a function $s=s^{\kappa,v}(c)$.
    \item If the slope $s_\xi/c_\xi$ is positive for some point $(s,c)$ then it strictly increases when $\kappa$ or $v$ increases.
\end{enumerate}
\end{proposition}

The following proposition describes the properties of trajectories that enter the saddle point $u^+_1(v)$ or leave the saddle point $u^-_2(v)$.

\begin{proposition}\label{prop_half_trajectories}
Consider one of the dynamical systems \eqref{eq:dyn_sys_cap_non-eq}, \eqref{eq:dyn_sys_cap_diff}. Then for all $v$ the following holds:
\begin{enumerate}[label=\Alph*)]
    \item If $u^-_2(v)$  exists and $u^-_1(v) \neq u^-_2(v)$, then there is a unique trajectory leaving $u^-_2(v)$ inside~$\Omega$.
    \item If $u^+_1(v)$ exists  and $u^+_1(v) \neq u^+_2(v)$, then there is a unique trajectory entering $u^+_1(v)$ inside~$\Omega$.
    \item When either trajectory exists: 
    \begin{itemize}
    \item it depends continuously on $\kappa$ and $v$ pointwise as a function $s^{\kappa,v}(c)$ on $(c^+, c^-)$;
    \item it is increasing as a function $s^{\kappa,v}(c)$ in some vicinity of the critical point; 
    \item it depends monotonously on $\kappa$ and $v$ pointwise as a function $s^{\kappa,v}(c)$ in some vicinity of the critical point: the trajectory from $u_2^-(v)$ decreases when $\kappa$ or $v$ increases, and the trajectory to $u_1^+(v)$ increases when $\kappa$ or $v$ increases.
    \end{itemize}
    \item Every existing trajectory from $u^-_2(v)$ to $u^+_1(v)$  is monotone as a function $s^{\kappa,v}(c)$ and has a positive slope on $(c^+, c^-)$.
\end{enumerate}
\end{proposition}

\section{Proof of Theorem~\ref{Theorem1} }
\label{sec:shock-admissibility}





 First, combining Remarks~\ref{rm:velocity-admiss} and~\ref{rm:critical_points} we see that a $c$-shock wave from $u^+$ to $u^-$ is compatible by speeds in a sequence of waves~\eqref{eq:solution-RP-1} if each of the points $u^-$ and $u^+$ is either a saddle point or a saddle-node of the travelling wave dynamical system \eqref{eq:dyn_sys_cap_non-eq} or \eqref{eq:dyn_sys_cap_diff}. Thus we will pay special attention to trajectories connecting two saddle points, i.e. $u^-_2(v)$ and $u^+_1(v)$.
 
 Second, the following Lemma shows that we are only interested in portraits of Type II and intermediate portraits of Type I-II, Type II-III and Type II-IV:
\begin{lemma}
\label{lm:TypeI-III-no-trajectory}
Type~0 and Type~IV do not have a pair of saddle points to connect. 
Type~I and Type~III 
do not have a trajectory connecting saddle points $u^-_2(v)$ and $u^+_1(v)$.
\end{lemma}
\begin{proof}
For Type I phase portrait there exists an interval $(s_1,s_2)\subset (s_1^+,s_2^-)$ such that 
\begin{align}
    \label{eq:s_c_derivative}
\dfrac{s_\xi}{c_\xi}= \kappa\dfrac{(v d_1) (f(s, c) - v( s + d_1))}{A(s,c) (d_1 c - d_2 - a(c))}<0 \quad\text{ on } (s_1,s_2)\times(c^+,c^-).
\end{align}
Thus any trajectory going through this rectangle must have $\frac{d}{dc}s^{\kappa,v}(c)<0$. 
For Type III phase portrait there exists an interval $(s_1,s_2)\subset (s_2^-,s_1^+)$ such that $s_\xi/c_\xi>0$ on $(s_1,s_2)\times(c^+,c^-)$.   Thus any trajectory going through this rectangle must have $\frac{d}{dc}s^{\kappa,v}(c)>0$. 
From these two facts  Lemma~\ref{lm:TypeI-III-no-trajectory} follows immediately (see Fig.~\ref{fig:lemma3}).
\end{proof}

\begin{figure}[H]
    \centering
    \includegraphics[width=0.49\textwidth]{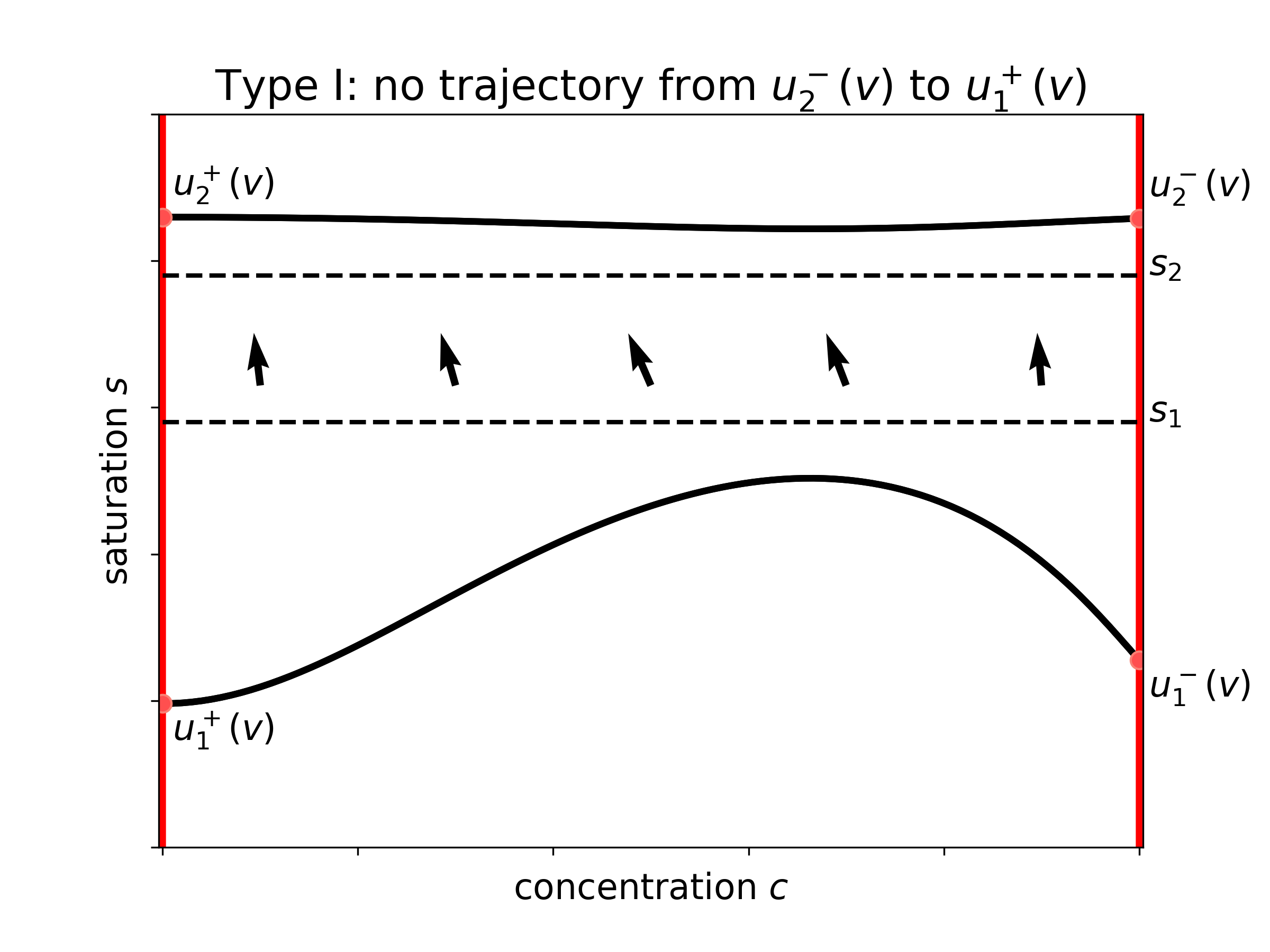}
    \includegraphics[width=0.49\textwidth]{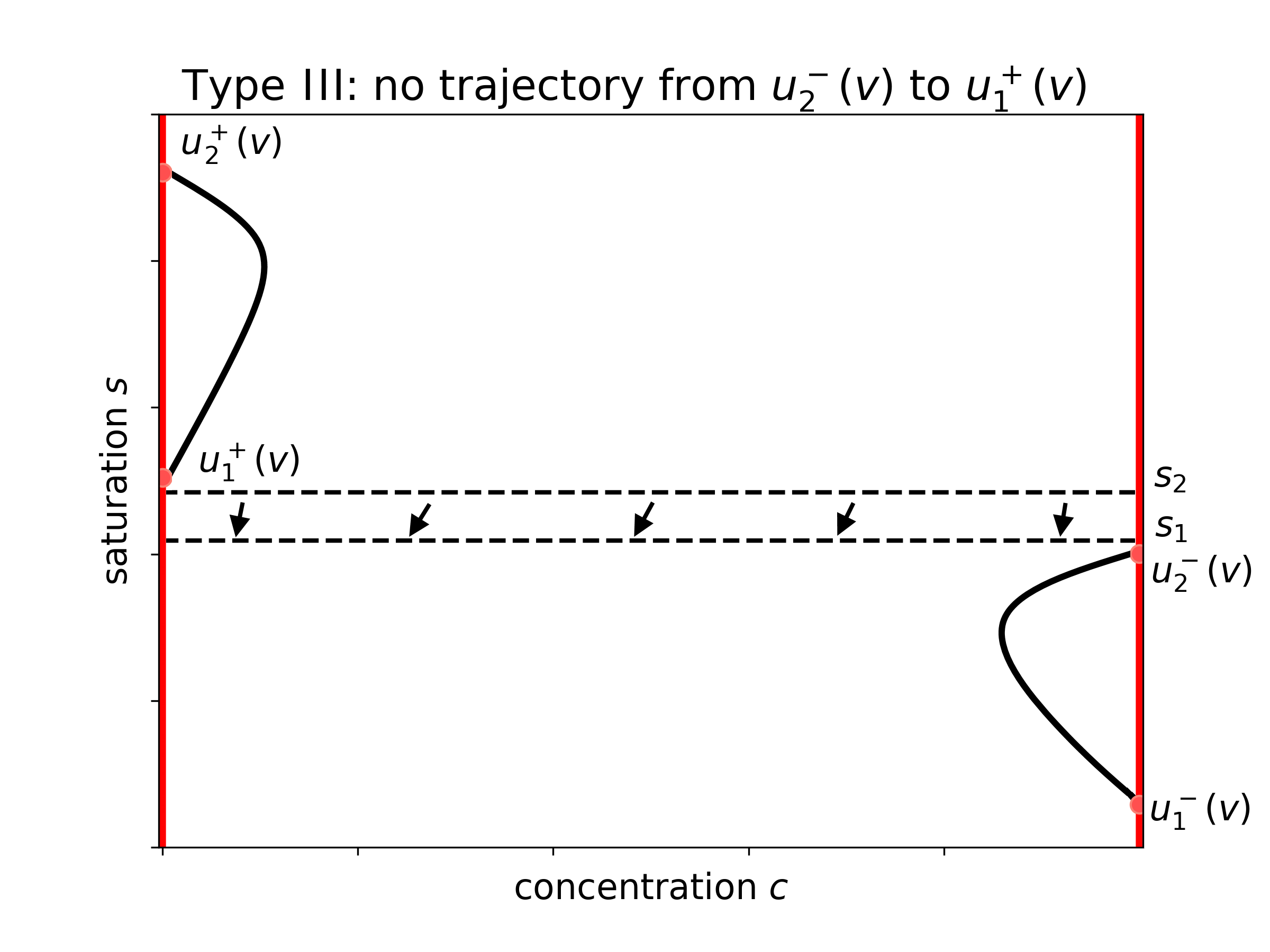}
    \caption{schematic proof of Lemma~\ref{lm:TypeI-III-no-trajectory}}
    \label{fig:lemma3}
\end{figure}

Third, we prove the following theorem for Type II phase portraits. Theorem~\ref{Theorem1} easily follows from it 
(for details see proof at the end of this section).


\begin{theorem}\label{Theorem2}
For every Type II phase portrait, i.e. for every $v\in(v_{\min}, v_{\max})$, there exists a unique $\kappa(v) \geq 0$ such that there exists a trajectory from $u_2^-(v)$ to $u_1^+(v)$. Moreover, $\kappa$ is monotone in $v$, continuous, $\kappa(v) \to \infty$ as $v\to v_{\min}$, and there exists $\kappa_{crit} \in [0, +\infty)$ such that $\kappa(v) \to \kappa_{crit}$ as $v \to v_{\max}$.
\end{theorem}

We  divide Theorem~\ref{Theorem2} into smaller lemmas.

\begin{lemma}
\label{lm:existence_kappa}
For every $v\in(v_{\min}, v_{\max})$, there exists $\kappa(v) \geq 0$ such that there exists a trajectory from $u_2^-(v)$ to $u_1^+(v)$.
\end{lemma}
\begin{proof}

First, by Proposition \ref{prop_half_trajectories} (points A, B) there exist a unique trajectory leaving $u_2^-(v)$ and a unique trajectory entering $u_1^+(v)$. We are looking for values of $\kappa$ for which these two trajectories intersect, and thus coincide.

Second, by the definition of Type II portrait there exists an interval $[c_1, c_2] \subset (c^+, c^-)$ where
\[
f(s, c) - v( s + d_1) < 0, \qquad \forall s\in(0,1), \forall c\in{[c_1,c_2]},
\]
and thus for some $\varepsilon>0$ we have
\begin{align*}
\kappa^{-1} \dfrac{s_\xi}{c_\xi} = \dfrac{(v d_1) (f(s, c) - v( s + d_1))}{A(s,c) (d_1 c - d_2 - a(c))} > \varepsilon > 0 \qquad \forall s\in(0,1), \forall c\in{[c_1,c_2]}.
\end{align*}
Therefore, there exists $\kappa_* = (\varepsilon \cdot (c_2 - c_1))^{-1}$, such that for all $\kappa > \kappa_*$ and any trajectory $s^{\kappa,v}(c)$
\[
\dfrac{d}{dc}s^{\kappa,v}(c) > \dfrac{1}{c_2 - c_1} \qquad  \forall c\in{[c_1,c_2]}.
\]
Now, if we consider the linear function 
\[
l_1(c) =  \dfrac{c - c_1}{c_2 - c_1},
\]
then the trajectory from $u_2^-(v)$ and the trajectory to $u_1^+(v)$ cannot intersect its graph, thus the trajectory from $u_2^-(v)$ goes below the trajectory to $u_1^+(v)$ (see Fig. \ref{fig:lemma_existence_kappa}a).

\begin{figure}[H]
    \centering
    \includegraphics[width=0.49\textwidth]{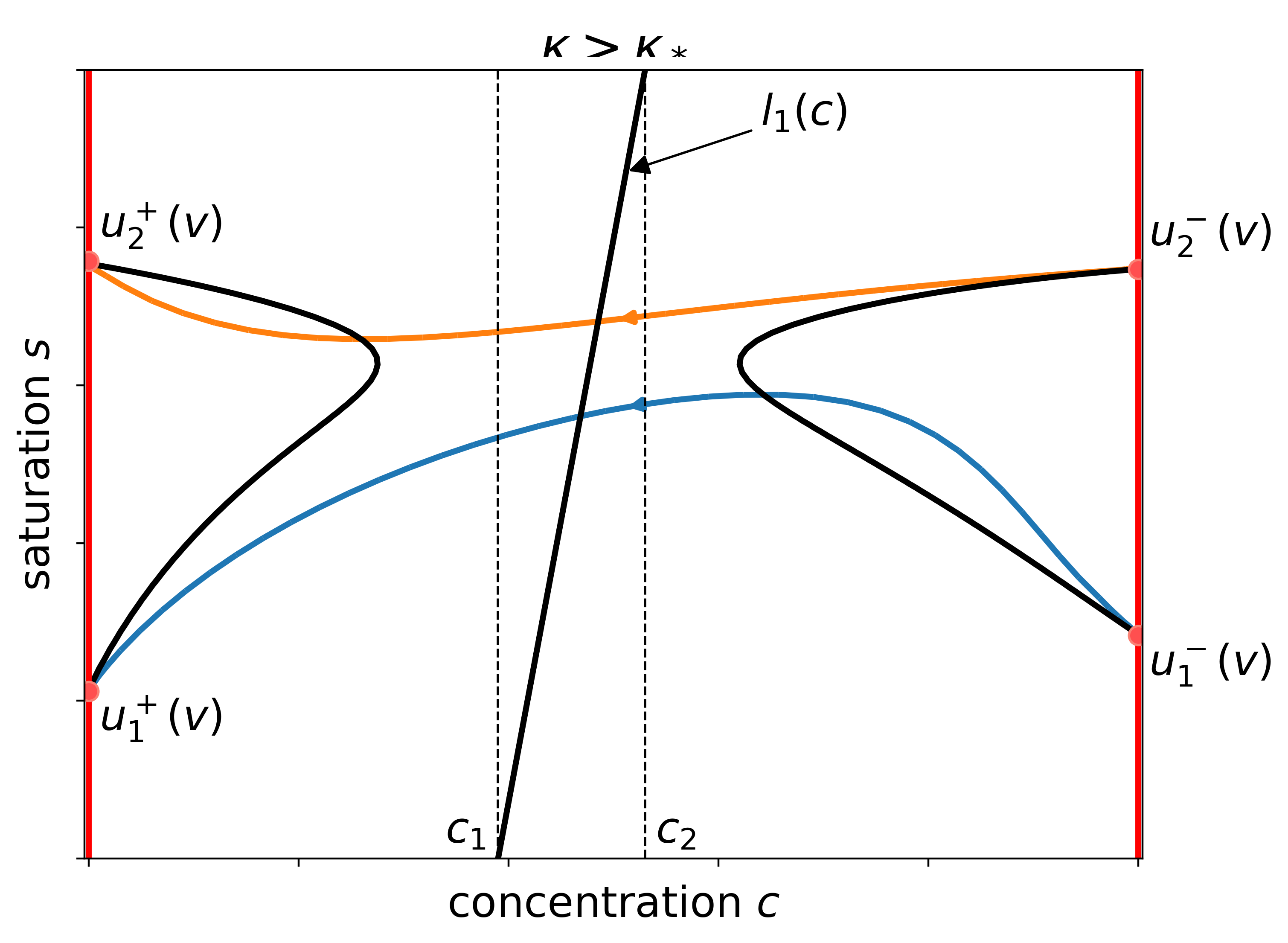}
    \includegraphics[width=0.49\textwidth]{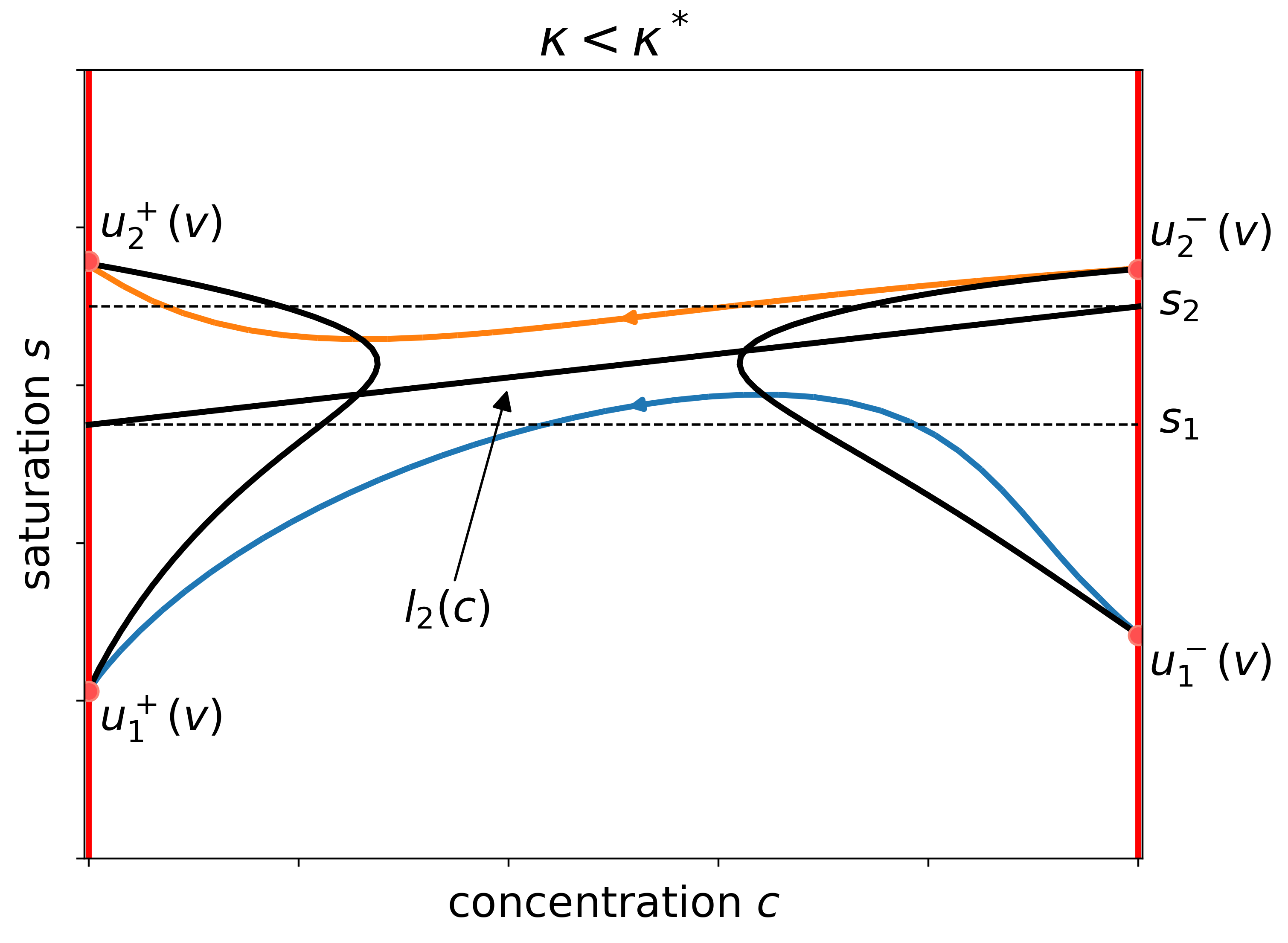}\\
    (a)\qquad\qquad\hfil\qquad\qquad (b)\hfil
    \caption{schematic proof of Lemma~\ref{lm:existence_kappa}}
    \label{fig:lemma_existence_kappa}
\end{figure}

Third, by the definition of Type II portrait we can fix a pair of points 
\[
s_1 \in (s_1^+(v), \min(s_2^+(v), s_2^-(v)) ), \quad 
s_2 \in (\max(s_1^+(v), s_1^-(v)), s_2^-(v) ), 
\]
and construct the linear function 
\[
l_2(c) =  \dfrac{s_2 - s_1}{c^- - c^+} (c - c^+) + s_1.
\]
Then on this line we have
\[
\kappa^{-1} \dfrac{s_\xi}{c_\xi} = \dfrac{(v d_1) (f(s, c) - v( s + d_1))}{A(s,c) (d_1 c - d_2 - a(c))} < M < +\infty \qquad 
\text{for all points } (l_2(c), c).
\]
Thus, there exists $\kappa^* = (M \cdot (c^- - c^+))^{-1} \cdot (s_2 - s_1)$, such that for all $\kappa < \kappa^*$ and for any trajectory $s^{\kappa,v}(c)$ intersecting the line $l_2(c)$ we have
\[
\dfrac{d}{d c}s^{\kappa,v}(c) < \dfrac{s_2 - s_1}{c^- - c^+} \qquad 
\text{for all points } (l_2(c), c).
\]
Thus, the trajectory from $u_2^-(v)$ and the trajectory to $u_1^+(v)$ cannot intersect its graph, and the trajectory from $u_2^-$ stays above the trajectory to $u_1^+$ (see Fig. \ref{fig:lemma_existence_kappa}b).

Finally, from continuous dependence on $\kappa$ (Proposition \ref{prop_half_trajectories}, point C) we conclude that there exists $\kappa \in [\kappa^*, \kappa_*]$ for which the trajectories coincide.
\end{proof}

\begin{lemma}
\label{lm:uniqueness_kappa}
$\kappa(v)$ is unique for every $v\in(v_{\min}, v_{\max})$.
\end{lemma}
\begin{proof}
Suppose there are $\kappa_1 < \kappa_2$ for which a trajectory exists. Due to Proposition \ref{prop_half_trajectories} (point C) we know that trajectories depend monotonically on $\kappa$ in the vicinities of the critical points, but the trajectories leaving $u_2^-(v)$ and entering $u_1^+(v)$ have the opposite monotonicity. Therefore, they must intersect at some $c\in(0,1)$.

Specifically, for trajectories $s^{\kappa_1,v}(c)$ and $s^{\kappa_2,v}(c)$ connecting the saddle points $u_2^-(v)$ and $u_1^+(v)$, we have
\[
s^{\kappa_1,v}(c) \leq s^{\kappa_2,v}(c)  \text{ near } u_1^+(v),
\]
\[
s^{\kappa_1,v}(c) \geq s^{\kappa_2,v}(c)  \text{ near } u_2^-(v).
\]
Thus, there exists a point $u_0=(s_0, c_0)$, such that 
\[
s^{\kappa_1,v}(c_0) = s^{\kappa_2,v}(c_0) = s_0, \quad \dfrac{d}{dc}s^{\kappa_1,v}(c_0) \geq \dfrac{d}{dc}s^{\kappa_2,v}(c_0)>0,
\]
and these slopes are positive due to Proposition~\ref{prop_half_trajectories} (point D) which contradicts Proposition~\ref{prop1} (point~E).

\begin{figure}[H]
    \centering
    \includegraphics[width=0.49\textwidth]{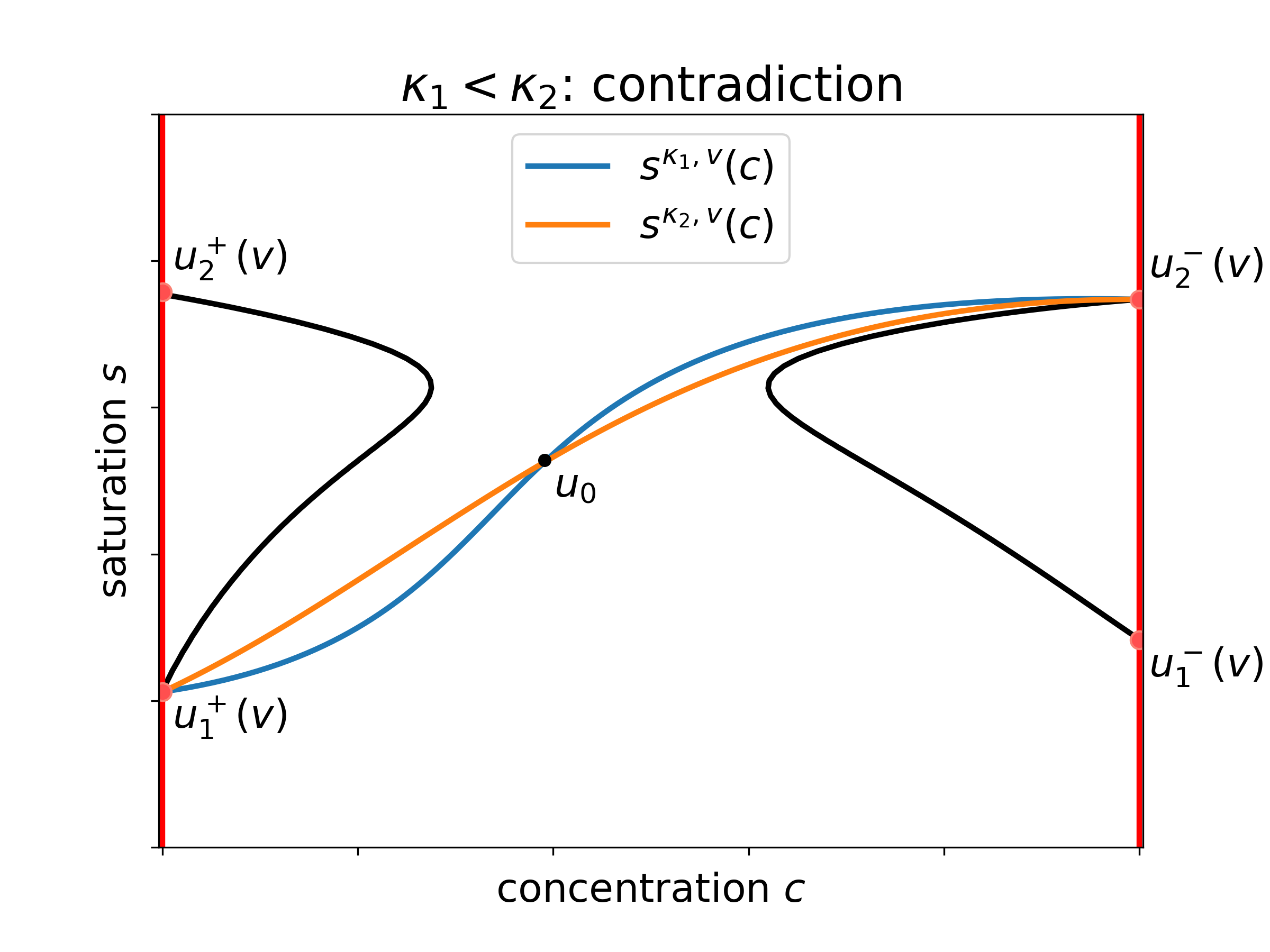}
        \includegraphics[width=0.49\textwidth]{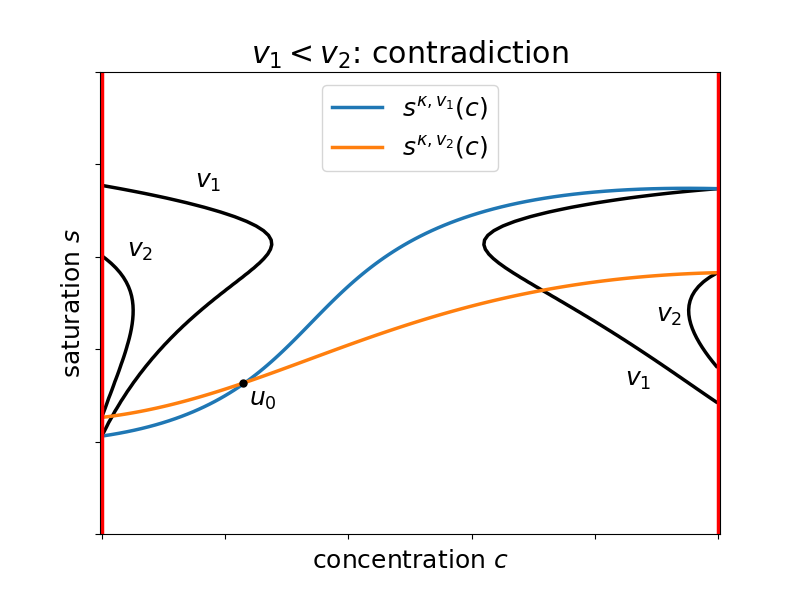}\\
    (a)\qquad\qquad\hfil\qquad\qquad (b)\hfil
    \caption{(a) schematic proof of  Lemma~\ref{lm:uniqueness_kappa}; (b) schematic proof of Lemma~\ref{lemma7}}
    \label{fig:lemma_uniqueness_kappa}
\end{figure}


\end{proof}

\begin{lemma}\label{lemma7}
For a fixed value of $\kappa>0$ there cannot exist more than one value of $v$, such that there exists a trajectory from $u_2^-(v)$ to $u_1^+(v)$. Thus, $\kappa$ is monotonous in $v$.
\end{lemma}
\begin{proof}

Similar to Lemma~\ref{lm:uniqueness_kappa}, suppose there are $v_1 < v_2$ for which a trajectory exists. Due to Proposition \ref{prop_half_trajectories} (point C) we know that trajectories depend monotonically on $v$ in the vicinities of the critical points, but the trajectories leaving $u_2^-(v)$ and entering $u_1^+(v)$ have the opposite monotonicity. 

Suppose 
$s^{\kappa,v_1}(c)$ and $s^{\kappa,v_2}(c)$ are the trajectories connecting saddle points. By Proposition \ref{prop_half_trajectories} (point~C)
we have
\[
s^{\kappa,v_1}(c) \leq s^{\kappa,v_2}(c)  \text{ near } c^+,
\]
\[
s^{\kappa,v_1}(c) \geq s^{\kappa,v_2}(c)  \text{ near } c^-.
\]
Thus, there exists a point $u_0=(s_0, c_0)$, such that 
\[
s^{\kappa,v_1}(c_0) = s^{\kappa,v_2}(c_0) = s_0, \quad \dfrac{d}{dc}s^{\kappa,v_1}(c_0) \geq \dfrac{d}{dc}s^{\kappa,v_2}(c_0) > 0, 
\]
and these slopes are positive due to Proposition~\ref{prop_half_trajectories} (point~D) which contradicts Proposition~\ref{prop1} (point~E).
\end{proof}

\begin{lemma}
Function $\kappa$ depends continuously on $v$.
\end{lemma}
\begin{proof}
Due to Proposition~\ref{prop_half_trajectories} (point~C) the trajectory depends continuously on both~$\kappa$ and~$v$, thus it creates a continuous dependence of~$\kappa(v)$. More rigorously, let $s_1^{\kappa, v}(c)$ be the trajectory ending in $u_1^+(v)$ and let $s_2^{\kappa, v}(c)$ be the trajectory beginning in $u_2^-(v)$. Then there exists a trajectory connecting $u_1^+(v)$ and $u_2^-(v)$ if and only if $s_1^{\kappa(v),v}(c) = s_2^{\kappa(v),v}(c)$. Now if we fix any point $c_0 \in (c^+, c^-)$ then $\kappa(v)$ is the solution of the implicit function equation $s_1^{\kappa(v),v}(c_0) - s_2^{\kappa(v),v}(c_0) = 0$. According to Proposition~\ref{prop_half_trajectories} (point~C) the function $h(\kappa, v) := s_1^{\kappa,v}(c_0) - s_2^{\kappa,v}(c_0)$ is continuous and strictly monotonous in~$\kappa$ in the vicinity of any given point $(\kappa(v_0), v_0)$, thus according to the implicit function theorem the equation $h(\kappa, v) = 0$ corresponds to the continuous solution $\kappa(v)$.
\end{proof}


\begin{figure}[H]
    \centering
    \includegraphics[width=0.49\textwidth]{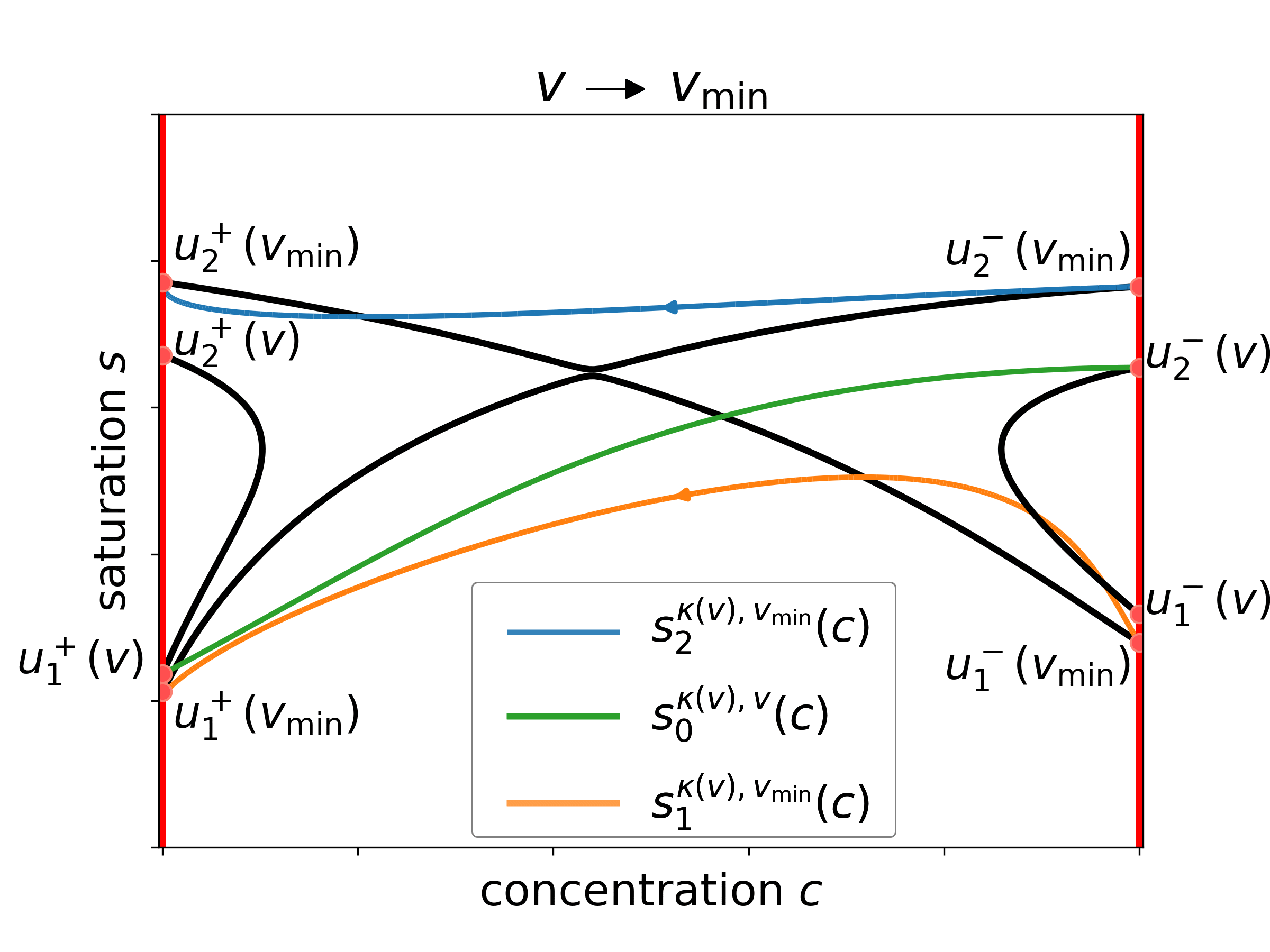}
    \caption{schematic proof of Lemma~\ref{lm:kappa_infinity}}
    \label{fig:lemma_kappa_infinity}
\end{figure}

\begin{lemma}
\label{lm:kappa_infinity}
$\kappa \to \infty$ as $v\to v_{\min}+0$.
\end{lemma}
\begin{proof}
Suppose there exists a finite limit $\kappa(v)\to \hat{\kappa}<\infty$ as $v\to v_{\min}$.
Similar to the proof of continuity,  let $s_1^{\kappa, v_{\min}}(c)$ be the trajectory ending in $u_1^+(v_{\min})$ and let $s_2^{\kappa, v_{\min}}(c)$ be the trajectory beginning in $u_2^-(v_{\min})$. 
Let $s_0^{\kappa,v}(c)$ be the trajectory connecting $u_1^+(v)$ and $u_2^-(v)$. Then according to Proposition \ref{prop_half_trajectories} (point~C) the trajectory $s_0^{\kappa,v}(c)$ is sandwiched between $s_1^{\kappa(v), v_{\min}}(c)$ and $s_2^{\kappa(v), v_{\min}}(c)$ and all three trajectories get closer as $v\to v_{\min}$ due to continuous dependence on $\kappa$ and $v$, i.e.
\[
\lim\limits_{v\to v_{\min}} s_1^{\kappa(v), v_{\min}}(c) - s_0^{\kappa(v), v}(c) = 0,\qquad \forall c\in(c^+,c^-),
\]
\[
\lim\limits_{v\to v_{\min}} s_2^{\kappa(v), v_{\min}}(c) - s_0^{\kappa(v), v}(c) = 0,\qquad \forall c\in(c^+,c^-),
\]
thus
\[
\lim\limits_{v\to v_{\min}} s_1^{\kappa(v), v_{\min}}(c) - s_2^{\kappa(v), v_{\min}}(c) = 0,\qquad \forall c\in(c^+,c^-),
\]
which cannot occur for any finite value of $\hat{\kappa}$, since $s_1^{\hat{\kappa}, v_{\min}}(c)$ goes strictly below the black line connecting $u_1^+(v_{\min})$ and $u_2^-(v_{\min})$, and $s_2^{\hat{\kappa}, v_{\min}}(c)$ goes strictly above the same line.


\end{proof}


When $v$ tends to $v_{\max}$ from below portrait of Type~II evolves into either portrait of Type II-III or Type II-IV.

\begin{lemma}
If we have Type II-III, then $\kappa \to 0$ as $v \to v_{\max}-0$.
\end{lemma}
\begin{proof}
As we approach the Type II-III, we have $s_2^-(v) - s_1^+(v) \to +0$. At the same time, for the trajectory $s(c)=s^{\kappa,v}(c)$, connecting saddle points we have
\begin{equation}
\label{eq:delta_s_integral}
s_2^-(v) - s_1^+(v) = \int\limits_{c^+}^{c^-} s'(c) \,dc = \kappa(v) \int\limits_{c^+}^{c^-} \dfrac{(v d_1) (f(s(c), c) - v( s(c) + d_1))}{A(s(c),c) (d_1 c - d_2 - a(c))} \, dc,
\end{equation}
and the integral is separated from zero (since the integrand is separated from zero everywhere including $c^+$ and $c^-$), thus $\kappa(v) \to +0$ as well.
\end{proof}

\begin{lemma}
\label{lm:v_max_type-2-3}
If we have Type II-IV, then there exists a positive finite limit $\kappa(v) \to \kappa_{crit}>0$ as $v\to v_{\max}-0$ and for every $0 < \kappa \leq \kappa_{crit}$ we have $v(\kappa) = v_{\max}$, i.e. for $v_{\max}$ there exists a trajectory from $u_2^-(v)$ to $u_1^+(v)$ for every $0<\kappa \leq \kappa_{crit}$.
\end{lemma}
\begin{proof}
The limit $\kappa_{crit}$ exists since $\kappa(v)$ is monotone and non-negative. Similar to the previous lemma, we utilise the relation \eqref{eq:delta_s_integral}, but this time both $s_2^-(v_{\max}) - s_1^+(v_{\max})$ and the integral in the right part are bounded and separated from zero, thus $\kappa_{crit} > 0$. The trajectories depend monotonously on $\kappa$, thus for every $0 < \kappa < \kappa_{crit}$ the trajectory arriving at $u_1^+(v_{\max})$ must be lower, than the same trajectory for $\kappa_{crit}$. At the same time, due to Proposition \ref{prop1} (point~C) every such trajectory must begin either in $u_1^-(v_{\max}) = u_2^-(v_{\max})$, or at $s = 1$. Thus, all such trajectories begin in $u_1^-(v_{\max}) = u_2^-(v_{\max})$.
\end{proof}
Together Lemmas~\ref{lm:existence_kappa}--\ref{lm:v_max_type-2-3} give the proof of Theorem~\ref{Theorem2}. Finally, we can derive Theorem~\ref{Theorem1} from Theorem~\ref{Theorem2}.

\begin{proof}[Proof of Theorem~\ref{Theorem1}]
Since $\kappa(v)$ given by Theorem~\ref{Theorem2} is a continuous monotone function, there exists an inverse function $v(\kappa)$ for $\kappa\geqslant \kappa_{crit}$. By Lemma \ref{lm:v_max_type-2-3} for $\kappa\in(0,\kappa_{crit})$ and $v=v_{\max}$ there exists a trajectory from $u_2^-(v)$ to $u_1^+(v)$. Thus Theorem~\ref{Theorem1} follows from Theorem~\ref{Theorem2} for any $\kappa>0$.
\end{proof}

\section{Discussions and generalizations}
\label{sec:generalisations}

\subsection{Milder assumptions on the flow function}
The aim of the paper is to provide a clear proof for the simplest non-monotone dependence of flow function $f$ on concentration $c$. 
However, in practice assumptions (F1)-(F4) are too rigorous and usually do not hold for real life data. Initially we wanted to allow a lot of weaker inequalities and other degrees of freedom in assumptions on $f$:
\begin{enumerate}
    \item[(F2*)] $f_s(s, c)\geq 0$ for $0<s<1$, $0 \leq c \leq 1$;  $f_s(0,c)=f_s(1,c)=0$;
    \item[(F3*)] $f$ is S-shaped in $s$: for each $c \in [0,1]$ function $f(\cdot,c)$ has a point of inflection $s^I =  s^I(c) \in [0, 1]$, such that $f_{ss}(s, c)\geq 0$ for $0<s<s^I$ and $f_{ss}(s, c)\leq 0$ for $s^I<s<1$. 
    \item[(F4*)] $f$ is possibly non-monotone in $c$: $\forall s \in (0,1) \, \exists c^*(s) \in [0,1]$:
    \begin{itemize}
        \item $f_c(s, c)\leq 0$ for $0<s<1$, $0 < c < c^*(s)$;
        \item $f_c(s, c)\geq 0$ for $0<s<1$, $c^*(s) < c < 1$.
    \end{itemize}
\end{enumerate}
Sadly, these assumptions include some functions that break part of the assertions of Theorem~\ref{Theorem1}. Most notably, the monotone case is included under these weaker assumptions. Since such functions~$f$ could be obtained as a limit of functions satisfying (F1)-(F4), we believe that the ``limiting'' variant of Theorem~\ref{Theorem1} should hold for them, i.e. at worst 
$v(\kappa)$ might become non-strictly monotonous, or 
we might have $v_{\min} = v_{\max}$, and thus trivial dependence $v(\kappa) = const$.

Our best recommendation for functions $f$ from this wider class is to look at the phase portrait sequence. If the phase portrait sequence generated by $f$ is similar enough to the ones we considered above (i.e. it at least has a Type II portrait in it), then Theorem~\ref{Theorem1} will most likely hold.

Additionally, any changes to $f$ below the line $l(s) = (v_{\min} - \epsilon) \cdot (s + d_1)$ for some $\epsilon > 0$ do not perturb the important part of the phase portrait sequence, and thus will not break the assertions of Theorem~\ref{Theorem1}.

Finally, we believe that some milder variants of 
Theorem~\ref{Theorem1} hold under fewer assumptions on~$f$. Specifically, in future work we aim to get rid of the assumptions (F3)-(F4). Certainly, the monotonicity of $v(\kappa)$ will no longer hold in this case. But nevertheless, for some functions $f$ in this class we intend to obtain non-trivial dependence of $v$ on $\kappa$.

\subsection{Dissipative system with three parameters}
\label{subsec_6_2}
Another direction of investigation is to consider the dissipative system~\eqref{eq:main_system_dissipation} with three dimensionless groups $\varepsilon_c, \varepsilon_d$ and $\varepsilon_r$, and study the corresponding travelling wave dynamical system:
\begin{equation}
\begin{cases} 
A(s,c) s_\xi = f(s, c) - v (s + d_1), \\
\kappa_1 c_\xi = v(d_1 c - d_2 - \alpha),\\
\kappa_2 \kappa_1 \alpha_\xi = v^{-1}(\alpha-a(c)).
\end{cases}
\end{equation}
Here $\kappa_1=\varepsilon_d/\varepsilon_c$ and $\kappa_2=\varepsilon_r/\varepsilon_d$. The increased dimension of the dynamical system makes the analysis a bit more complicated in comparison to systems~\eqref{eq:dyn_sys_cap_non-eq} and~\eqref{eq:dyn_sys_cap_diff}. However, the last two equations are decoupled from the first one, and thus could be solved separately to reduce the dimension. The system 
\begin{equation*}
\begin{cases} 
 c_\xi = v(d_1 c - d_2 - \alpha),\\
\kappa_2 \alpha_\xi = v^{-1}(\alpha-a(c)).
\end{cases}
\end{equation*}
has two critical points $(0, a(0))$ (focus) and $(1, a(1))$ (saddle point), thus there is a unique solution $\alpha = \alpha(c)$ connecting them. We can substitute $\alpha(c)$ into the second equation and reduce the problem to the previous case: 
\begin{equation*}
\begin{cases} 
A(s,c) s_\xi = f(s, c) - v (s + d_1), \\
\kappa_1 c_\xi = v(d_1 c - d_2 - \alpha(c)).
\end{cases}
\end{equation*}
For a fixed $\kappa_2$ we get the same dependence of $v(\kappa_1)$ as in Theorem~\ref{Theorem1}. The dependence of $v$ on $\kappa_2$ could be the subject of further studies. We believe that $\alpha(c)$ depends monotonously on $\kappa_2$ and thus it should be possible to prove the monotonous dependence of $v$ on $\kappa_2$, but that requires more rigorous arguments. 

It would be interesting to consider a truly three-equation dynamical system (e.g. for a three-phase model including gas or relaxation time dependent on flux), but that is a much more complex problem that we do not know how to approach yet.

\subsection{General Riemann problem}
Note that we considered one particular Riemann problem $(1,1)\to(0,0)$ as we wanted to focus more on travelling wave solutions corresponding to a unique $c$-shock wave connecting the saddle points of the travelling wave dynamical system. We believe that for the non-monotone case the result analogous to~\cite{JohansenWinther} may be proved: any Riemann problem $(s^L, c^L)\to(s^R, c^R)$ has a unique solution for a fixed value of $\kappa$. For $c_L<c_R$ it is true as there are no $c$-shock waves; for $c_L>c_R$ a more careful analysis of admissible shocks is needed, since different Riemann problems might allow connections other than saddle to saddle to be compatible by speed.





\subsection{Stability and convergence}
When a Riemann problem solution is found, a lot of stability and convergence questions arise: 
\begin{itemize}
    \item Does the Riemann problem solution of the dissipative system converge to the non-dissipative solution as $\varepsilon_{c}\to 0$ with fixed $\kappa$ or $\kappa_1$, $\kappa_2$?
    \item Are undercompressive travelling wave solutions stable in the sense that any solution of a Cauchy problem for the dissipative system with the correct values at $\pm\infty$ tends to the admissible travelling wave as $t\to+\infty$?
\end{itemize}
There are similar results for a single equation (see \cite{MejaiVolpert}, \cite{Gasnikov}), for strictly hyperbolic systems of conservation laws (see \cite{GoodmanXin}) and for undercompressive shocks in certain hyperbolic systems of conservation laws (see \cite{LiuZumbrun}). Note that there are examples of unstable undercompressive shocks~\cite{GardnerZumbrun}. To our knowledge, for our systems these questions are still open even in the monotone case.
\begin{appendices}

\section{}
\label{ap:A}

\begin{proof}[Proof of Proposition \ref{prop1}]
A) Functions $f$ and $a$ are $C^2$-smooth, thus the right-hand side of the dynamical system is Lipschitz. The  point A thus follows from the Picard theorem and its corollaries.

B) The system is autonomous, so the point B follows from the uniqueness of the solution.

C) Note the following:
\begin{itemize}
    \item The right-hand side of the second equation is signed:
    \[
    d_1c - d_2 - a(c) < 0, \qquad \forall c\in(c^+, c^-),
    \]
    thus every trajectory goes strictly from right to left. This also implies the point D.
    \item The border lines $\{c=c^\pm\}$ also consist of orbits for the system and critical points. Since orbits cannot intersect, other orbits can only approach these border lines asymptotically as $\xi \to \pm\infty$. Consider a point $(s, c^+)$, at which the trajectory approaches the boundary as $\xi \to +\infty$. If $(s, c^+)$ is not critical, then $s_\xi/c_\xi$ is signed on this trajectory in the $\varepsilon$-vicinity of $(s, c^+)$ for some $\varepsilon>0$ and
    \[
    \left| \dfrac{s_\xi}{c_\xi} \right| > \dfrac{const}{c - c^+}, \quad c\in(c^+, c^+ + \varepsilon), 
    \]
    as in formula \eqref{eq:s_c_derivative} the enumerator is not zero and denominator is equivalent to $K(c-c^+)$ for $K=d_1-a'(c^+)\not=0$ due to strict concavity of $a$. This inequality means that $s$ cannot be bounded as $c\to c^+$ and leads to a contradiction. Thus, $(s,c^+)$ can only be a critical point.
\end{itemize}

E) Consider the dynamical system \eqref{eq:dyn_sys_cap_non-eq}. We recall the relation \eqref{eq:s_c_derivative} and rewrite it in the following way:
\[
\dfrac{s_\xi}{c_\xi}= \kappa\dfrac{d_1 (v^2( s + d_1) - vf(s, c))}{A(s,c) ( a(c) - d_1 c + d_2)} > 0.
\]
Since this expression is positive, it increases when $\kappa$ increases. And when $v$ increases, $v^2(s+d_1) - vf(s,c)$ increases, since
\[
\dfrac{d}{dv} \Big(v^2(s+d_1) - vf(s,c)\Big) = 2v(s+d_1) - f(s,c) > v(s+d_1) - f(s,c) > 0,
\] 
thus the slope increases when $v$ increases as well.

\end{proof}

\begin{remark}\label{remark_difference_between_systems}
In the case of the system \eqref{eq:dyn_sys_cap_diff} instead of $(v^2(s+d_1) - vf(s,c))$ we have $((s+d_1) - v^{-1} f(s,c))$ in the numerator of the relation \eqref{eq:s_c_derivative}. It also increases when $v$ increases, thus the proof of E does not change.
\end{remark}


\begin{proof}[Proof of Proposition \ref{prop_half_trajectories}]
A, B) The existence and the uniqueness of such trajectories follows immediately from the Hartman-Grobman theorem, as $u_2^-(v)$ and $u_1^+(v)$ are saddle points (see Remark~\ref{rm:critical_points}).

C) To obtain the continuous dependence on $\kappa$ and $v$, let us fix $c_0 \in (c^+, c^-)$. Denote by $s_1^{\kappa, v}(c)$ the trajectory arriving at $u_1^+(v)$ and by $s_\theta^{\kappa, v}(c)$ the trajectory going through $(\theta, c_0)$ for any $\theta \in (0,1)$. It is clear that due to Proposition~\ref{prop1} every trajectory $s_\theta^{\kappa, v}(c)$ for $\theta>s_1^{\kappa, v}(c_0)$ must end in $u_2^+(v)$, and every trajectory $s_\theta^{\kappa, v}(c)$ for $\theta<s_1^{\kappa, v}(c_0)$ must end in $\{s=0\}$. Thus, the trajectory entering $u_1^+(v)$ is the supremum of all trajectories below it. More rigorously, we write
\begin{equation}\label{s_1_as_supremum}
s_1^{\kappa, v}(c) = 
\sup \Big\{s_\theta^{\kappa,v}(c) : s_\theta^{\kappa, v} \text{ ends at the border } s=0  \Big\}.
\end{equation}
Every trajectory $s_\theta^{\kappa, v}(c)$ satisfies all the Picard theorem corollaries assumptions, and thus depend continuously on $\kappa$, $v$ and $\theta$. Thus, $s_1^{\kappa, v}(c)$, as a supremum, also depends continuously on $\kappa$ and $v$. 

The trajectory $s_1^{\kappa, v}(c)$ is an increasing function in the vicinity of the critical point $u_1^+(v)$, since any trajectory with a decreasing slope close enough to the border $\{c=c^+\}$ will end in $u_2^+(v)$.

Now, to obtain the monotonous dependence in the vicinity of the critical point, 
note that due to previous point $\frac{d}{dc} s_1^{\kappa, v}(c) = s_\xi/c_\xi > 0$ in some vicinity $(c^+, c^++\varepsilon)$, $\varepsilon>0$. For this part of the proof we will assume $c_0\in(c^+, c^++\varepsilon)$. Then for all trajectories $s_\theta^{\kappa, v}(c)$ below $s_1^{\kappa, v}(c)$ we also have $\frac{d}{dc} s_\theta^{\kappa, v}(c) > 0$ on $(c^+, c^++\varepsilon)$.
Due to Proposition~\ref{prop1} (point~E), when $\kappa$ or $v$ increases, the slope of all $s_\theta^{\kappa, v}(c)$ below $s_1^{\kappa,v}(c)$ increases. Therefore, the set $\{s_\theta^{\kappa,v}(c) : s_\theta^{\kappa, v} \text{ ends at the border } s=0  \}$ can only expand, and the supremum in \eqref{s_1_as_supremum} increases.

Similar arguments give the proof for the trajectories leaving $u_2^-(v)$.

D) If it's not monotone, then it could be crossed three times for some value of $s$. More rigorously, let there be a point $(s_0, c_0)$ such that the trajectory $s(c) = s^{\kappa, v}(c)$ connecting $u_2^-(v)$ and $u_1^+(v)$ is not monotone:
\[
s'(c_0) = \kappa \dfrac{(v d_1) (f(s_0, c_0) - v( s_0 + d_1))}{A(s_0,c_0) (d_1 c_0 - d_2 - a(c_0))} < 0.
\]
Then $s$ takes the value $s_0$ at least two more times (before and after $c_0$) and has a non-negative derivative at those points. Thus, for some $c_1 < c_0 < c_2$ we have
\[
f(s_0, c_0) - v( s_0 + d_1)) > 0,
\]
\[
f(s_0, c_{1,2}) - v( s_0 + d_1)) \leqslant 0,
\]
which contradicts the property (F4) of $f$ (see Section \ref{subsec:restrictions}).

Since the function is monotone, we conclude that it's slopes are non-negative. All that's left is to note, that if a trajectory with non-negative slopes touches the nullcline (the black curve) inside $\Omega$ on any portrait, it must cross it. Thus, the slopes cannot be zero and therefore are positive on $(c^+, c^-)$.
\end{proof}

\end{appendices}

\section*{Acknowledgements}

We are grateful to Pavel Bedrikovetsky and Sergey Tikhomirov for fruitful discussions. Research is supported by Russian Science Foundation grant 19-71-30002.

\end{document}